\documentclass{article}
\usepackage{amsmath,amssymb, amsthm,float,mathrsfs}
\usepackage[dvipdfmx]{graphicx,xcolor}
\usepackage{comment}
\usepackage{authblk}

\newcommand{\norm}[1]{\lVert  #1 \rVert}

\newcommand{\HP}{\widehat{\mathbb{P}}}
\newcommand{\HE}{\widehat{\mathbb{E}}}
\newcommand{\CE}{\mathcal{E}}
\newcommand{\CF}{\mathcal{F}}

\renewcommand{\P}{\mathbb{P}}

\renewcommand{\epsilon}{\varepsilon}
\newcommand{\OH}{\widehat{\Omega}}

\renewcommand{\phi}{\varphi}
\newcommand{\holeind}{\gamma}
\newcommand{\distind}{\Upsilon}
\newcommand{\Holderind}{\alpha}
\newcommand{\Sobolevind}[1]{\frac{d#1}{d-#1 \zeta}}
\newcommand{\innerproduct}[2]{\langle #1, #2\rangle}
\newcommand{\Rtheta}{\hat{R}_\theta}

\theoremstyle{definition}
\newtheorem{Def}{Definition}[section]

\newtheorem{eg}[Def]{Example}

\newtheorem{asm}{Assumption}
\newtheorem{asmR}{Assumption}

\theoremstyle{plain}
\newtheorem{thm}[Def]{Theorem}
\newtheorem{lem}[Def]{Lemma}

\newtheorem{prop}[Def]{Proposition}
\newtheorem{cor}[Def]{Corollary}

\makeatletter

\@addtoreset{equation}{section}
\makeatother

\makeatletter
\newcommand{\biggg}{\bBigg@{3}}
\newcommand{\bigggg}{\bBigg@{4}}
\makeatother

\begin{document}
\title{Local Central Limit Theorem for Reflecting Diffusions in a Continuum Percolation Cluster}
\author{Yutaka \textsc{takeuchi}\thanks{
    \begin{minipage}[t]{\textwidth}
        Department of Mathematics, Faculty of Science and Technology, Keio University

        e-mail: \texttt{yutaka.takeuchi@keio.jp}
    \end{minipage}
}}
\date{}
\maketitle

 \begin{abstract}
Reflecting diffusions on continuum percolation clusters are considered.
Assuming that the occupied region has a unique unbounded cluster and the cluster satisfies geometrical conditions such as volume regularity, isoperimetric conditions, and a hole size condition,  we prove a quenched local central limit theorem for reflecting diffusions on the cluster.
 \end{abstract}

\section{Introduction and result}
Markov processes on random media have been studied by not only mathematicians but scientists of the other fields. One of the important problems is homogenization. 
Loosely speaking, this is a process of replacing an equation with highly oscillatory coefficients by one with homogeneous coefficients.  Kipnis and Varadhan \cite{KV} proved an annealed invariance principle for random walk on a supercritical (bond) percolation cluster. In the 2000s, many authors studied the quenched invariance principle for random walk on the supercritical percolation cluster  (\cite{SS}, \cite{BB}, \cite{MP}). Quenched invariance principle for versions of more general random conductance model have been shown by many authors (\cite{A}, \cite{ABDH}, \cite{BD}, \cite{BP}, \cite{M}). Recently, the quenched invariance principle for random conductance model with more general assumptions was shown (\cite{ADS}, \cite{DNS}, \cite{BS1}). 
Some researchers showed local limit theorems, which are more precise results. Barlow and Hambly \cite{BH} proved the local central limit theorem for random walks on the supercritical bond percolation cluster. Later, Andres and Taylor \cite{AT} and Bela and Sch\"{a}ffner \cite{BS2} proved the local central limit theorem for random walks on the random conductance model. 
In a continuum setting, Chiarini and Deuschel proved both a quenched invariance principle and a local central limit theorem for diffusions in a degenerate random environment (\cite{CD2}, \cite{CD}).

The work presented below mainly concerns a homogenization on the continuum percolation built over stationary ergodic point processes.  
In the 1990s, Tanemura \cite{T}, Osada \cite{O2}, and Osada and Saitoh \cite{OS} proved the annealed invariance principle for reflecting diffusions in continuum percolation clusters.
Recently, the author proved a quenched invariance principle for reflecting diffusions in modified continuum percolation clusters \cite{Y}. In this paper we discuss the local central limit theorem for the process. To explain, we first introduce the configuration space on the $d$-dimensional Euclidean space $\mathbb{R}^d$. That is, 
\[\Omega = \Biggl \{\sum _{i=1}^\infty \delta _{x_i}  \Biggm| \ \sum _{i=1}^\infty \delta _{x_i}(K) <\infty \text{ for each compact set }K \subset \mathbb{R}^d, x_i \in \mathbb{R}^d  \Biggr \} .\]
Here $\delta_{{x}_i}$ is the Dirac measure that places unit mass at $x_i$.  The space $ \Omega $ is equipped with the vague topology. Under this topology, $\Omega$ is a Polish space. (See \cite[Theorem 4.2]{Ka}.) The Borel $ \sigma $-field $ \mathcal{B}(\Omega) $ is generated by the set $ \{ \omega \in \Omega \mid \omega(A) = n\} $, $ A \in \mathcal{B}(\mathbb{R}^d) , n \in \mathbb{N}$.  

Let $\rho >0$ be a fixed constant. Set 
\[\Omega_0 = \Biggl\{\omega  \in \Omega \Biggm| \begin{aligned}	& |x-y|\neq 2\rho \;\;  (\text{for any } x,y \in \text{supp}\;\omega ) \\ &\text{ and $\omega(x)= 1$ for all $x \in \text{supp}\;\omega $}\end{aligned}   \Biggr\}.\] 
We identify each configuration $\omega = \sum _{i=1}^\infty \delta _{x_i} \in \Omega $ with a countable subset $ \{x_i\}_i $ of $ \mathbb{R}^d $ if each $ x_i $ is distinct.

We introduce continuum percolation. For $ \omega \in \Omega \setminus \Delta $, define the subset $ L_\rho(\omega) $ of $ \mathbb{R}^d $ by
\[ L_\rho(\omega) =\bigcup _{x\in \omega}B_\textup{Euc}(x,\rho),\]
where $B_\textup{Euc}(x,\rho)$ is the Euclidean open ball with center $x\in \mathbb{R}^d$ and radius $\rho$.
Define the critical radius $\rho_c$ by
\begin{align*}
	\rho_c = \inf\{\rho > 0 \mid \P (L_\rho(\omega) \text{ has unbounded component containing the origin}) > 0 \},
\end{align*}
where we used the convention $\inf \emptyset = \infty$.
Let $W(\omega) $ be the unbounded connected component of $L(\omega)$ if there is a unique unbounded component. Otherwise we set $ W(\omega)=\emptyset $ by convention. We call $ W(\omega) $ a percolation cluster. It is a continuum analogue of a discrete site-percolation cluster (See \cite{G}). By definition, $W(\omega)$ can be written as $\bigcup_{x \in I_\rho(\omega)}B_\textup{Euc}(x,\rho)$ for some subset $I_\rho(\omega)$ of $\omega$. We also take $\rho' \geq \rho$ and we introduce the modified cluster $W'(\omega)$ given by
\begin{align*}
	W'(\omega) = \bigcup_{x \in I_\rho(\omega)}B_\textup{Euc}(x,\rho'). 
\end{align*}
In this paper, we will discuss reflecting diffusions on the modified cluster $W'(\omega)$. An advantage of treating $W'(\omega)$ instead of $W(\omega)$ is that we do not need to handle ``traps''. When we consider the percolation cluster $W(\omega)$, it can have arbitrary narrow `bottlenecks'. Such bottlenecks may influence the long-time behavior of the reflecting diffusion. In particular, it is known that the existence of traps can be an obstacle to a quenched invariance principle. Indeed, Barlow, Burdzy, and Tim\'{a}r (\cite{BBT1}, \cite{BBT2}) constructed a random environment such that the associated random walk satisfies an annealed invariance principle but not a quenched invariance principle. The percolation cluster $W(\omega)$ may be such a model. However, this is a difficult problem that we leave open for further study. 
	
Define a subset $ \OH $ of $ \Omega $ by
\[\OH=\{\omega \in \Omega_0 \mid  0 \in  W'(\omega) \}.\]
Let $\P $ be a probability measure on $\Omega$. 
Set $\HP (\cdot )=\P (\cdot  \cap  \OH)/\P(\OH)$ if $\P(\OH) > 0$ and denote its expectation by $ \HE $. 
We define the shift  $\tau _z\colon \Omega \to \Omega$ by
\[\tau _z \omega (A) = \omega (A+z)=\sum_{x\in \omega} \delta_{x}(A+z)=\sum_{x\in \omega} \delta_{x-z}(A),\]
where $A+z = \{x+z \mid x \in A\}$ for $z \in \mathbb{R}^d$. We impose the following assumptions for the measure $\P$.

\begin{asm}\label{asm:erg} The probability measure $ \P $ satisfies the following conditions.
	\begin{enumerate} 
		\item $\P$ is stationary and ergodic with respect to $\{\tau _x\}_{x\in \mathbb{R}^d}$.
		\item $\P (\OH)>0$ and $ \P(\Delta)=0 $. 
	\end{enumerate}
\end{asm}

Thanks to Assumption \ref{asm:erg}$(2)$, $W'(\omega)$ is a Lipschitz domain for a.e. $\omega \in \OH$ since the number of balls in $W'(\omega) $ intersecting each compact set is finite by definition of the configuration space $\Omega$. 
The existence of a unique unbounded component is an important problem in the study of percolation theory and there are many studies that focus on this topic. If $ \P $ is a Poisson point process and the radius $\rho $ is bigger than the critical radius $\rho_c$, it is known that there is a unique unbounded component, see \cite{MR}.

We denote the Euclidean inner product by $ \langle \cdot , \cdot \rangle $. 
	Let $ a\colon \Omega \to \mathbb{R}^{d\times d} $ be a symmetric matrix-valued random variable. 
	\begin{asm}\label{asm:ellipse}
		There exist constants $\lambda,\Lambda > 0$ such that for $\HP$-almost all $\omega$
		\begin{align*}
			\lambda |\xi|^2 \leq \langle a(\omega)\xi, \xi \rangle \leq \Lambda |\xi|^2, \;\; \forall\xi \in \mathbb{R}^d.
		\end{align*}
	\end{asm}
	Define the bilinear form $ \CE^\omega $ on $L^2(W'(\omega),dx)$ by
	\begin{align}
	\CE^\omega (u,v)=\int_{W'(\omega)} \langle a(\tau_x\omega) \nabla u(x), \nabla v(x) \rangle dx.
	\end{align}
	We denote the closure of a subset $A$ by $\overline{A}$.
	Let $\CF^\omega$ be the completion of $C_c^\infty(\overline{W'(\omega)})$ with respect to $\CE^\omega(\cdot,\cdot)+\norm{\cdot}_{L^2(W'(\omega),  dx)}$.
	According to \cite{FT} and \cite{FT2}, under Assumptions \ref{asm:erg} and \ref{asm:ellipse}, $(\CE^\omega,\, \CF^\omega)$ is a strongly local and regular Dirichlet form  $\HP$-almost surely.   Hence, we have the associated conservative diffusion $\{X_t^\omega\}_t$. It is called a reflecting diffusion since the domain $\CF^\omega$ corresponds to the Neumann boundary condition. When we consider the case that $a(\omega)=\frac{1}{2}I_d$, $\{X_t^\omega\}_t$ is  reflecting Brownian motion. 
	We further impose an assumption for such reflecting diffusions.
	\begin{asm}\label{asm:density}
		The reflecting diffusion $\{X_t^\omega\}_t$ has a transition density $p_t^\omega(\cdot,\cdot)$ with respect to $dx$ for $\HP$-almost all $\omega$.
	\end{asm}

Next, we impose  geometric conditions that play important roles. 
Let $W$ be a Lipschitz domain.
For $x,y \in W$, define the distance $d_W(x,y)$ between $x$ and $y$ to be the infimum of the length of piecewise smooth paths in $W$ connecting $x$ and $y$.
Define an ``intrinsic" open ball $B_W(x,R)$ by $\{y \in W \mid d_W(x,y) < R\}$ for $x \in W, R > 0$. We denote the Euclidean distance by $d_\text{Euc}(x,y) = |x-y| $. We recall that the Euclidean open ball is written $B_\textup{Euc}(x,R)=\{y \in \mathbb{R}^d
\mid d_\text{Euc}(x,y) < R\}$. We first introduce a notion of a good ball.

\begin{Def}\label{def:reg}
	Let $W$ be a Lipschitz domain.
    \begin{enumerate}
      \item   Let $x \in W$, $R>0$, $C_\textup{V}\geq 1$, $C_\textup{iso}>0$. We say that a ball $B_W(x,R)$ is ($C_\textup{V}, C_\textup{iso}$)-regular  if the following conditions hold.
        \begin{enumerate}
            \item Denoting  the $d$-dimensional Lebesgue measure by $|\cdot |$, 
			\begin{equation}\label{cdn:volumeRegular}
                C_\textup{V}^{-1}R^d \leq |B_W(x,R)| \leq  C_\textup{V}R^d.
            \end{equation}  
            \item Denote the $(d-1)$-dimensional Hausdorff measure by $\mathcal{H}_{d-1}$.  For every open subset $O \subset B_W(x,R)$ that has a Lipschitz boundary and satisfies \hspace{20pt} $|O| \leq \frac{1}{2}|B_W(x,R)|$,
			\begin{equation}\label{cdn:relativeIso}
                \mathcal{H}_{d-1}(B_W(x,R)\cap \partial O) \geq C_\textup{iso}R^{-1}|O| . 
            \end{equation} 
        \end{enumerate}  
		\item Let $\theta \in (0,1)$, $C_\textup{V}\geq 1$, $C_\textup{iso}>0$. We say $W$ is $(C_\textup{V}, C_\textup{iso}, \theta)$-very regular if there exists $\Rtheta > 0$ such that for all $R\geq \Rtheta$, the ball $B_W(x,r)$ is $(C_\textup{V}, C_\textup{iso})$-regular for all $x\in B_W(0,R)$ and $r \geq R^\frac{\theta}{d}$. 
    \end{enumerate}
\end{Def}



\begin{asm}\label{asm:volIso}
	\begin{enumerate}
		\item ($\theta$-very regularity (cf. \cite[Assumption 1.3]{DNS}))
		For $\HP$-a.s.\,$\omega$, there exist $C_\textup{V}^\omega \geq 1$,  $C_\textup{iso}^\omega > 0$, and
		$ \theta^\omega \in (0,1) $ such that  $W'(\omega)$ is $(C_\textup{V}^\omega, C_\textup{iso}^\omega, \theta^\omega)$-very regular.
		\item (Isoperimetric condition) For $\HP$-a.s.\,$\omega$, there exists a constant $c_H^\omega>0$ such that
		\begin{gather}\label{cdn:IPforShortInPerocolationCluster}
			C_\textup{IS}^\omega:=\inf\bigggg\{\frac{\mathcal{H}_{d-1}(W'(\omega)\cap \partial O)}{|W'(\omega)\cap O|^\frac{d-1}{d}}\bigggg| \begin{split} &O \subset W'(\omega)\text{ is bounded open}\\ &\text{with Lipschitz boundary,}\\ &\mathcal{H}_{d-1}(W'(\omega)\cap \partial O) < c_H^\omega
			\end{split}\bigggg\} > 0.
		\end{gather}
	\end{enumerate}	
\end{asm}

Intuitively, Assumption \ref{asm:volIso} (1) guarantees the regularity of the cluster at the macroscopic scale. Assumption \ref{asm:volIso} (2) tells us that the cluster has no narrow bottlenecks (and hence there is no trap).

	Once we construct a diffusion process, we can consider whether some limit theorem holds or not. One of the important questions is whether an invariance principle holds or not. In \cite{Y}, the author proved the following media-wise invariance principle (quenched invariance principle).

	\begin{thm}[Quenched invariance principle, \cite{Y}]\label{thm:qip}
		Let $ d \geq 2 $. Suppose that Assumptions \ref{asm:erg}-\ref{asm:volIso} hold. Let $ P_0^\omega $ be the law of $ \{X_t^\omega\}_t$ starting at $ 0 $. 
		Then for $\HP$-a.s.\,$\omega$, the scaled process $\{\epsilon X_{\epsilon ^{-2}t}^\omega\}_t$ under $ P_0^\omega $ converges in law to a Brownian motion with deterministic non-degenerate covariance matrix $\Sigma $ as $\epsilon$ tends to $0$.
	\end{thm}

	To prove the local central limit theorem, we need an assumption about size of a ``hole" in the component. 
	For $R>0$ and a Lipschitz domain $W$, define a maximal hole size $h_W(R)$ by 
\begin{align*}
	h_W(R) = \sup\{ d_\text{Euc}(x, W) \mid x \in B_\textup{Euc}(0,R) \}.
\end{align*}

\begin{asm}\label{asm:holedistOfCluster}
	For $\HP$-a.s.\,$\omega$, the following statements hold.
	\begin{enumerate}
		\item There exists $\holeind \equiv \holeind(\omega) \in (0,1)$ such that
		\begin{align}
			\lim_{R\to \infty} \frac{h_{W'(\omega)} (R)}{R^{\holeind}} = 0.
		\end{align}
		\item There exist $\distind \equiv \distind(\omega) \in (0, 1)$ and $C_W^\omega$ such that 
		\begin{align}
			d_{W'(\omega)(x,y)} \leq C_W d_\text{Euc}(x,y) \vee R^{\distind}
		\end{align}
		for $x,y \in B^\omega(0,R)$ and $R\geq \Rtheta^\omega$, where $\Rtheta^\omega$ is the constant which appeared in Assumption \ref{asm:volIso}.
	\end{enumerate}
\end{asm}

Assumption \ref{asm:holedistOfCluster} $(2)$ says that the intrinsic distance can compare with the Euclidean distance. 
For a Lipschitz domain $W$, Let $G_W(x)$ be the set of closest points to $x$ in $\overline{W}$. Let $g_W(x)$ be the first element of $G_W(x)$ using the lexicographic order. Here, for $x=(x_1,\dots, x_d), y=(y_1,\dots, y_d) \in \mathbb{R}^d$, we say that $x$ is less than $y$ in the lexicographic order if $x_1 < y_1$ or there exists $j \in \{2,\dots, d\} $  such that $x_i = y_i$ for $i= 1,\dots, j-1$ and $x_j < y_j$.   We denote $g_{W'(\omega)}(x)=g^\omega(x)$.
Let $ k_t^\Sigma$ be the Gaussian kernel with covariance matrix $ \Sigma $, namely
\begin{align*}
k_t^\Sigma(x) = \frac{1}{\sqrt{(2\pi t)^d \det \Sigma } } \exp \biggl(- \frac{\langle x, \Sigma^{-1} x \rangle}{2t} \biggr).
\end{align*}

The main theorem of this paper is the following.

\begin{thm}[Local central limit theorem]\label{thm:localCLT}
	Let $ d \geq 2 $. Suppose that Assumptions \ref{asm:erg}--\ref{asm:holedistOfCluster} hold.
	 Let $R > 0$ and $I \subset (0,\infty)$ be a compact interval. Then for $\HP$-almost all $\omega \in \OH$, 
\begin{align*}
    \lim_{\epsilon \to 0}\sup_{|x| < R}\sup_{t \in I} \Bigl|\epsilon^{-d}p_{t/\epsilon^2}^\omega(0,g^\omega(x/\epsilon)) - k_t^\Sigma(x)  \Bigr| = 0.
\end{align*}
\end{thm}


\begin{eg}[Reflecting Brownian motion on a Poisson Boolean model]
	Let $\P$ be a Poisson point process. We take radius $\rho$ greater than the critical radius $\rho_c$ so that there is a unique unbounded connected component of $L_\rho(\omega)$. Set $a(\omega) = \frac{1}{2}I_d$. Then corresponding diffusion is the reflecting Brownian motion. We can easily verify that Assumption \ref{asm:erg} and \ref{asm:ellipse} are satisfied. 
	By results in \cite{Ma}, Assumption \ref{asm:density} is satisfied.
	Now we need to check Assumptions \ref{asm:volIso} and \ref{asm:holedistOfCluster}. In the previous paper \cite{Y}, Assumption \ref{asm:volIso} is essentially proved in Example 1.3.  
	 For Assumption \ref{asm:holedistOfCluster}, Barlow and Hambly \cite[Lemma 5.4.]{BH} showed that the Bernoulli percolation cluster satisfies Assumption {\ref{asm:holedistOfCluster}}  with $\gamma = 1/2$. For a Poisson Boolean model, from this and the comparison argument as in \cite[Example 1.3]{Y}, Assumption \ref{asm:holedistOfCluster} holds. 
\end{eg}

 A key ingredient to prove the local central limit theorem is the parabolic Harnack inequality. This is obtained by considering a solution of the following equation:
\begin{align}
	\partial_tu(t,x) = \frac{1}{\Lambda}\nabla \cdot (a(\tau_x\omega)\nabla u(t,x)), \quad t \in \mathbb{R},\; x \in W'(\omega).
\end{align}
To deduce the parabolic Harnack inequality, we show the following two inequalities;
\begin{align*}
	\begin{cases}
		\sup_{Q_-}|u| \leq C_1 \norm{u}_{\alpha,Q_\sigma},\\
		\norm{u}_{\alpha,Q_\sigma} \leq C_2 \inf_{Q_+}|u|.
	\end{cases}
\end{align*}
Here $Q_-, Q_\sigma$, and $Q_+$ are subsets of $\mathbb{R}\times W'(\omega)$ defined in $(\ref{def:paraball2})$, $(\ref{def:paraball})$, and $(\ref{def:paraballplus})$ and $\alpha >1$ is a suitable exponent and $\norm{u}_{\alpha, Q_\sigma}$ is the $L^\alpha$ norm averaged over $Q_\sigma$.
In \cite{CD}, Chiarini and Deuschel showed local central limit theorems for diffusions in  general ergodic environments.  They used Moser's iteration scheme to show the above inequalities. In our settings, Assumption \ref{asm:volIso} is needed to apply Moser's method.
Once we obtain the parabolic Harnack inequality, we can get the H\"{o}lder continuity of the density $p_t^\omega$. Thanks to this, we can show the local central limit theorem.

In \cite{CD},  Chiarini and Deuschel used analytical inequalities such as the Sobolev inequality to show the parabolic Harnack inequality. Our setting is similar to that of \cite{CD}, but with a reflecting boundary, which introduces additional difficulty. For instance, we need to consider the boundary behavior of reflecting diffusions, the shape of the cluster, and sizes of holes.  In terms of the associated Dirichlet form,  its domain is different to that of \cite{CD}. Hence, it is not trivial to deduce whether related analytical inequalities hold or not. For example, to show the Sobolev inequality, the isoperimetric condition is essentially used (see \cite[Proposition 3.3]{Y}). Indeed, the Sobolev inequality and the isoperimetric inequality has a relationship in general. (See \cite[Section 3]{K}.)
For the proof of the quenched invariance principle, we decompose the process into a martingale part and a remainder part that is called a  corrector part. The martingale part converges to the Brownian motion with the covariance matrix $\Sigma$ by a martingale convergence theorem. To show that the corrector vanishes after taking a scaling limit, we show a maximal inequality, which leads to the elliptic Harnack inequality. Since the quenched invariance principle is a point-wise estimation, it doesn't need the uniform estimation in time.
On the other hand, the local central limit theorem needs a time-space estimation such as the parabolic Harnack inequality and the H\"{o}lder continuity of the density. So, we need a more complicated analysis. 

The remainder of the article is organized as follows. In Section 2, we consider a deterministic setting and prove the parabolic Harnack inequality. In Section 3, we will prove the local central limit theorem in this deterministic setting. Then we will apply the above result to the continuum percolation model.

\section{Parabolic Harnack inequality and H\"{o}lder continuity for a deterministic model} 
In this section, we prove the parabolic Harnack inequality and the H\"{o}lder continuity of the density $p_t^\omega$. Since the main result is a local central limit theorem for quenched environment, it is enough to consider a deterministic setting. 
Let $W$ be a Lipschitz domain containing the origin. 


\begin{asmR}\label{asm:volRegiso}
	\begin{enumerate}
		\item ($\theta$-very regularity)
		There exist $C_\textup{V} \geq 1$, $C_\textup{iso} > 0$, and $ \theta \in (0,1) $ such that $W$ is $(C_\textup{V}, C_\textup{iso},\theta)$-very regular.
		\item (Isoperimetric condition) There exists a constant $c_H$ such that
		\begin{gather}\label{cdn:IPforShort} 
			C_ \text{IS}:=\inf\bigggg\{\frac{\mathcal{H}_{d-1}(W\cap \partial O)}{|W\cap O|^\frac{d-1}{d}}\bigggg| \begin{split} &O \subset W\text{ is bounded open}\\ &\text{with Lipschitz boundary,}\\ &\mathcal{H}_{d-1}(W\cap \partial O) < c_H
			\end{split}\bigggg\} > 0.
		\end{gather}
		\end{enumerate}
\end{asmR}
Let $ a\colon \mathbb{R}^d \to \mathbb{R}^{d\times d} $ be a symmetric matrix. We impose the following condition for the matrix $a$.

\begin{asmR}\label{asm:Dir}
		There exist positive constants $ \lambda, \Lambda > 0$ such that for Lebesgue almost every $x \in W$ and all $ \xi \in \mathbb{R}^d$
		\[ \lambda |\xi|^2 \leq \langle a(x)\xi, \xi \rangle \le \Lambda |\xi|^2.\] 
\end{asmR}
Define a bilinear form $ \CE $ on $L^2(W,  dx)$ by setting
\begin{align}
\CE (u,v)=\int_W \langle a \nabla u, \nabla v \rangle dx.
\end{align}
Define the domain $ \CF$ of the above bilinear form to be the completion of $ C_c^\infty(\overline{W}) $ with respect to $ \CE(\cdot,\cdot)+\norm{\cdot}_{L^2(W, dx)} $.

\begin{asmR}\label{asm:regdensity}
	The Dirichlet form $(\CE, \CF)$ is regular. 
	Moreover, the diffusion process $\{X_t\}_t$ associated with the Dirichlet form $(\CE, \CF)$ has a Lebesgue measurable density $p_t(x,y) $ with respect to $dx$.
\end{asmR}

\begin{asmR}\label{asm:holedist}
	\begin{enumerate}
		\item There exists $\holeind \in (0,1)$ such that
		\begin{align}\label{eqn:hole}
			\lim_{R\to \infty} \frac{h_W(R)}{R^{\holeind}} = 0.
		\end{align}
		\item There exist $\distind \in (0, 1)$ and $C_W$ such that 
		\begin{align}
			d_W(x,y) \leq C_W d_\text{Euc}(x,y) \vee R^\distind
		\end{align}
		for $x,y \in B_W(0,R)$ and $R\geq \Rtheta$.
	\end{enumerate}
\end{asmR}
Observe that under $(\ref{eqn:hole})$, we have that there exists the constant $C_\text{hole} >0$ such that for all $R\geq 1$, we have
\begin{align}
	h_W(R) \leq C_\text{hole}R^{\holeind}.
\end{align} 
We make the following assumption.

\begin{asmR}[Invariance principle]\label{asm:invariance}
	Let $\{X_t\}_t$ be the diffusion that appears in Assumption \ref{asm:Dir}. 
	There exists a positive definite symmetric matrix $ \Sigma $ such that the scaled process $\{\epsilon X_{\epsilon ^{-2}t}\}_t$ under $ P_0$ converges in law to Brownian motion with non-degenerate covariance matrix $\Sigma$ as $\epsilon$ tends to $0$. Here $P_0$ is the law of $\{X_t\}_t$ starting at $0$.		
\end{asmR}

Let $E\subset W$ be a subdomain with Lipschitz boundary. Define the bilinear form $ \CE_E $ on $L^2(E, dx)$ by
\[\CE_E(u,v) = \int_{E} \langle a \nabla u, \nabla v \rangle dx\]
and denote the completion of the set $\mathcal{C}_E(W) = \{u \in C_c^\infty(\overline{E}) \mid u = 0 \text{ in } W \setminus E\}$ with respect to $ \CE_{E}(\cdot,\cdot)+\norm{\cdot}_{L^2(E, dx)} $ by $ \CF_{E}$.

Throughout the paper, we use the following averaged norms. For each $\alpha \in (0, \infty)$, nonempty bounded open set $E \subset W$, and pair of functions $u,\phi$ on $E$ satisfying $\phi \geq 0$, set
\begin{align*}
	&\norm{u}_{\alpha,E} = \Biggl(\frac{1}{|E|} \int_{E} |u|^\alpha dx \Biggr)^\frac{1}{\alpha} \text{ and }
	\norm{u}_{\alpha, E, \phi} = \biggl( \frac{1}{|E|} \int_E |u|^\alpha \phi \,dx\biggr)^\frac{1}{\alpha}.
\end{align*}

\subsection{Sobolev inequality and Poincar\'{e} inequality}
In \cite{CD}, Chiarini and Deuschel deduced various inequalities from the classical Sobolev inequality and the Poincar\'{e} inequality. 
In this section, we explain these inequalities.

Deuschel, Nguyen and Slowik (\cite{DNS}) showed the relative isoperimetric inequality implies the isoperimetric inequality for large sets in the graph setting. In our setting, the same result holds.

\begin{prop}
	Assume that $W$ satisfies Assumption \ref{asm:volRegiso}. Let $\Rtheta$ be a constant appearing in Definition \ref{def:reg}\;(2). Then there exists a constant $C_\textup{IL}>0$ such that the following holds for all $R \geq \Rtheta$: 
	\begin{equation}
		\mathcal{H}_{d-1}(W \cap \partial O) \geq C_\textup{IL} |O|^\frac{d-1}{d} 
	\end{equation}
	for every open subset $O \subset B_W(0,R)$ that has Lipschitz boundary and satisfies $|O| \geq R^\theta$.
\end{prop}
\begin{proof}
	The proof is the same as \cite[Lemma 3.3]{DNS}. (The only difference is that we use the Lebesgue measure and the Hausdorff measure.)
\end{proof}

Set $\zeta = \frac{1-\theta}{1-\frac{\theta}{d}}$. In the previous paper \cite{Y}, the author proved the following weak isoperimetric inequality holds for a bounded open subset. 

\begin{lem}(\cite[Lemma 3.1.]{Y})\label{lem:WI} Suppose that Assumption \ref{asm:volRegiso} holds.
    Then there exists a positive constant $C_\textup{I}$ such that 
    \begin{align}\label{ine:wiso}
        \frac{\mathcal{H}_{d-1}(W\cap \partial O)}{|O|^\frac{d-\zeta}{d}}\geq \frac{C_\textup{I}}{R^{1-\zeta}}
    \end{align}
    holds for every open subset $O \subset B_W(0,R)$ with Lipschitz boundary.
\end{lem}

It is known that the isoperimetric inequality implies the Sobolev inequality. (See \cite[Theorem 3.2.7]{K} for example.) In the previous paper \cite{Y}, the author proved that the weak isoperimetric inequality deduces the following Sobolev-type inequality. 
\begin{lem}\label{lem:Sob1} (\cite[Proposition 3.2.]{Y})
	Suppose that Assumption \ref{asm:volRegiso} holds. 
    Fix $ R\geq \Rtheta$. Then for a ball $B=B_W(x_0,r)$ with $x_0 \in B_W(0, R)$ and $r \geq R^\frac{\theta}{d}$,
    we have 
    \begin{align}\label{ine:sobM}
        \norm{u}_{L^\frac{d}{d-\zeta}(B)} \leq C_\textup{Sob} |B|^\frac{1-\zeta}{d}\norm{\nabla u}_{L^1(B)}, \;\; u \in \CF_B,
    \end{align}
    where $C_\textup{Sob} = C_\textup{I}^{-1}C_\textup{V}^\frac{1-\zeta}{d}$.
\end{lem}

The next lemma is proved in \cite{Y}. 
\begin{lem}\label{lem:Sob}(\cite[Proposition 3.3.]{Y})
	Suppose that Assumption \ref{asm:volRegiso} holds. Let $B\subset W$ be an intrinsic ball as in Lemma \ref{lem:Sob1}.  Then for $1\leq p \leq 2$ we have 
	\begin{align}\label{ine:Sob}
		\norm{u}_{L^\frac{dp}{d-p\zeta}(B)} \leq C_\textup{Sob}|B|^\frac{1-\zeta}{d}\norm{\nabla u}_{L^p(B)}, \;\; u \in \CF_E.
	\end{align}
\end{lem}

Recall that $ \norm{u}_{\alpha,E} = (\frac{1}{|E|} \int_{E} |u|^\alpha dx )^\frac{1}{\alpha} $.
We have the following local Sobolev inequality.

\begin{lem}[Local Sobolev inequality]\label{lem:localSobolev}
	Suppose that Assumptions \ref{asm:volRegiso} and \ref{asm:Dir} hold.  Fix $ R\geq \Rtheta$. Then for a ball $B=B_W(x_0,r)$ with $x_0 \in B_W(0, R)$ and $r \geq R^\frac{\theta}{d}$,
	\begin{align}\label{ine:locSob}
	&\norm{u}_{\frac{dp}{d-p\zeta},B}^2 \leq C_\textup{Sob}^2 \lambda^{-1}|B|^{\frac{2-d}{d}}  \CE_B(u,u) 
	\end{align}
	for $1\leq p < 2$ and $u \in \CF_B$, where $\lambda$ is the constant appearing in Assumption \ref{asm:Dir}.
\end{lem}
\begin{proof}
		By the Sobolev-type inequality $(\ref{ine:Sob})$ and the H\"{o}lder inequality, we have
		\begin{align*}
			\norm{u}_{L^\frac{dp}{d-p\zeta}(B)}^2 &\leq C_\textup{Sob}^2|B|^\frac{2(1-\zeta)}{d} \norm{\nabla u}_{L^p(B)}^2 \\
			&=C_\textup{Sob}^2|B|^\frac{2(1-\zeta)}{d} \Biggl(\int_B |\nabla u|^p\lambda^\frac{p}{2} \cdot \lambda^{-\frac{p}{2}} dx\Biggr)^\frac{2}{p}\\
			&\leq C_\textup{Sob}^2|B|^\frac{2(1-\zeta)}{d} \Biggl( \biggl( \int_B |\nabla u|^2\lambda dx\biggr)^\frac{p}{2} \biggl( \int_B \lambda^{-\frac{p}{2}\frac{2}{2-p}} dx\biggr)^\frac{2-p}{2} \Biggr)^\frac{2}{p}\\
			&= C_\textup{Sob}^2\lambda^{-1}|B|^{\frac{2(1-\zeta)}{d} + \frac{2-p}{p}} \int_B|\nabla u|^2 \lambda dx. 
		\end{align*}
		
		By using this inequality and Assumption \ref{asm:Dir}, we have
		\begin{align*}
			\norm{u}_{L^\frac{dp}{d-p\zeta}(B)}^2 \leq  C_\textup{Sob}^2\lambda^{-1}|B|^{\frac{2(1-\zeta)}{d} + \frac{2-p}{p}} \CE_B(u,u).
		\end{align*}
		Therefore, we get the desired result after taking the average of the left-hand side. Note that the relation $\frac{2(1-\zeta)}{d} + \frac{2-p}{p}-\frac{2(d-p\zeta)}{dp} = \frac{2-d}{d}$. 
\end{proof}

For a bounded set $B$ and $u \in \CF_B$, denote the average of $u$ in $B$ by $(u)_B = |B|^{-1}\int_B u dx$.
It is known that the relative isoperimetric inequality implies the following $(1,1)$-Poincar\'{e} inequality.

\begin{lem}\label{lem:1-1Poincare}(\cite[Theorem 1.1]{KL})
	Suppose that Assumption \ref{asm:volRegiso} holds. Let $R \geq \Rtheta$ and $x \in B_W(0,R)$. Then there exists $\tilde{C}_\textup{P} > 0$ depending only on $C_\textup{iso}$ such that for $B=B_W(x,r)$ with $r \geq R^\frac{\theta}{d}$ and $u \in W^{1,1}(B)$, 
	\begin{align}
		\norm{u - (u)_B}_{1,B} \leq \tilde{C}_\textup{P} r \norm{\nabla u}_{1,B}.
	\end{align}
\end{lem}

A simple application of Lemma \ref{lem:1-1Poincare} gives the following $(2,2)$-Poincar\'{e} inequality.   

\begin{prop}\label{prop:Poincare}
	Suppose that Assumption \ref{asm:volRegiso} holds. Let $R \geq \Rtheta$ and $x \in B_W(0,R)$. Then for $B=B_W(x,r)$ with $r \geq R^\frac{\theta}{d}$ and $u \in \CF_B$, 
	\begin{align}
		\norm{u - (u)_B}_{2,B}^2 \leq C_\textup{P} r^2 \norm{\nabla u}_{2,B}^2,
	\end{align}
	where $C_\textup{P} = 4\tilde{C}_\textup{P}^2$.
\end{prop}
\begin{proof}
	We first note that $\inf_{a \in \mathbb{R}} \norm{f - a}_{q,B}\leq \norm{f - (f)_B}_{q, B} \leq 2\inf_{a \in \mathbb{R}} \norm{f - a}_{q,B}$ for $f \in \CF_B$ and $q \geq 1$. (It is well known. See the proof of Theorem 4.3 (ii) in \cite{KK} for example.) Without loss of generality, we can assume that $(u)_B = 0$. Set $v = u^2$. For $a\in \mathbb{R}$, we have
	\begin{align*}
		\norm{u-a}_{2,B}^2
		&= \int_B |u-a|^2 dx \\
		&=\int_B (u^2 -ua)dx - \int_B (ua -a^2) dx \\
		&=\int_{u\geq a} (u^2 -ua)dx + \int_{u < a} (u^2 -ua) dx\\
		& - \int_{u \geq a} (ua -a^2) dx - \int_{u < a} (ua -a^2) dx \\
		&\leq \int_{u\geq a} (u^2 -ua)dx - \int_{u < a} (ua -a^2) dx,
	\end{align*}
	where we used $ua - a^2 < 0$ on $\{u < a\}$ and $-(ua - a^2) \leq 0$ on $\{u \geq a\}$ in the last line. Since $u^2 - ua \leq u^2 - a^2$ on $\{u \geq a\}$ and $ua - a^2 > u^2 - a^2$ on $\{u < a\}$, the right-hand side of the above inequality is estimated above by 
	\begin{align*}
		\int_{u\geq a} (u^2 -a^2)dx - \int_{u < a} (u^2 -a^2) dx 
		&=  \int_{u\geq a} (u^2 -a^2)_+dx + \int_{u < a} (u^2 -a^2)_- dx \\
		&\leq \int_B (u^2 -a^2)_+dx + \int_B (u^2 -a^2)_- dx \\
		&= \int_B |u^2 -a^2|dx \\
		&= \norm{v- a^2}_{1,B}. 
	\end{align*}
	Therefore, we have 
	\begin{align*}
		\norm{u -a}_{2,B}^2 \leq  \norm{v- a^2}_{1,B}.
	\end{align*}
	Since $v$ is nonnegative, $(v)_B \geq 0$ and thus the infimum of $\norm{v - a}_{1,B}$ is attained by the nonnegative number $(v)_B$. Therefore, we have $\inf_{a \in \mathbb{R}}\norm{v - a^2}_{2,B}^2  \leq \norm{v - (v)_B}_{1,B}$. Recall that we assume $(u)_B=0$. Taking infimum of the above inequality we obtain
	\begin{align*}
		\norm{u}_{2,B}^2 \leq \norm{v - (v)_B}_{1,B}.
	\end{align*}
	Using Lemma \ref{lem:1-1Poincare} and the Cauchy-Schwartz inequality, we have
	\begin{align*}
		\norm{u}_{2,B}^2 \leq  \norm{v- a^2}_{1,B} \leq  \tilde{C}_\textup{P} r \norm{\nabla v}_{1,B} =\tilde{C}_\textup{P} r \norm{2u \nabla u}_{1,B} \leq \tilde{C}_\textup{P} r \norm{u}_{2,B}\norm{\nabla u}_{2,B}.
	\end{align*}
	Dividing both sides by $\norm{u}_{2,B}$ and taking squares, we obtain the desired result.
\end{proof}

\begin{prop}\label{prop:locPoi}
	Suppose that Assumptions \ref{asm:volRegiso} and \ref{asm:Dir} hold. Let $R \geq \Rtheta$. Let $r \geq R^\frac{\theta}{d}$ and $x \in B_W(0,R)$. Set $B = B_W(x,r)$.
	Then 
	\begin{equation}
		\norm{u - (u)_B}_{2,B}^2 \leq C_\textup{P}C_\textup{V}^{-\frac{2}{d}}\lambda^{-1} |B|^\frac{2-d}{d} \CE_B(u,u)
	\end{equation}
	for $u \in \CF_B$.
\end{prop}
\begin{proof}
	By Proposition \ref{prop:Poincare}, we have
	\begin{align*}
		\norm{u - (u)_B}_{2,B}^2 &\leq C_\textup{P} r^2 \frac{1}{|B|}\int_B |\nabla u|^2 dx \\
		&= C_\textup{P} r^2 \frac{1}{|B|}\int_B |\nabla u|^2\lambda \lambda^{-1} dx \\
		&= C_\textup{P} r^2 \lambda^{-1} \frac{1}{|B|}\int_B |\nabla u|^2\lambda dx.
	\end{align*}
	By Assumptions \ref{asm:volRegiso} and \ref{asm:Dir}, the last expression is bounded above by 
	\begin{align*}
		& C_\textup{P} C_\textup{V}^{-\frac{2}{d}}|B|^\frac{2}{d} \lambda^{-1} \frac{1}{|B|}\int_B \langle a\nabla u, \nabla u \rangle dx,
	\end{align*}
	which is the desired result. 
\end{proof}

\subsection{Cutoff inequalities}
We say that a function $u$ belongs to $\CF_\textup{loc}$ if for all intrinsic ball $B,$ there exists a function $u_B \in \CF_B$ such that $u = u_B$ in $B$.
By a cutoff function on $B_W(x_0, R)  $, we mean a function $ \eta \in C_c^\infty(\overline{B_W(x_0,R)}) $ such that $ 0 \leq \eta \leq 1 $ and vanishes in $W - B_W(x_0, R)$.
For a cutoff function $ \eta $ on $ B_W(x_0,R) $ and functions $ u,v \in \CF_\textup{loc}$, define
	\[\CE_\eta(u,v)= \int_{B_W(x_0,R)}\langle a \nabla u , \nabla v \rangle \eta^2 dx.\]

\begin{prop}
	Suppose that Assumptions \ref{asm:volRegiso} and \ref{asm:Dir} hold.  Let $R \geq \Rtheta$ and $B = B_W(x_0, r)\subset W$ be an intrinsic ball with $x_0 \in B_W(0,R)$ and $r \geq R^\frac{\theta}{d}$. Let $\eta \in C_c^\infty(\overline{B})$ be a cutoff  function on $B$. 
	Then for $p \in [1,2)$,
	\begin{equation}\label{ine:CutoffSobolev}
		\norm{\eta u}_{\frac{dp}{d-p\zeta}, B}^2 \leq 2C_\textup{S} |B|^\frac{2}{d}\Biggl(\frac{
		\CE_\eta(u,u)}{|B|} + \Lambda\norm{\nabla\eta}_\infty^2 \norm{u}_{2, B}^2 \Biggr)
	\end{equation} 
	holds for $u \in \CF_\textup{loc}$, where $C_\textup{S} = C_\textup{Sob}^2\lambda^{-1}$.
\end{prop}
\begin{proof}
    By the local Sobolev inequality $(\ref{ine:locSob})$, we have
    \begin{equation}\label{ine:inprop:cutoffSob}
        \norm{\eta u}_{\Sobolevind{p},B}^2 \leq C_\textup{Sob}^2\lambda^{-1}|B|^\frac{2}{d}\frac{\CE_B(\eta u,\eta u)}{|B|}.
    \end{equation}
    On the other hand, using Assumption \ref{asm:volRegiso}, we estimate
    \begin{align*}
        \CE_B(\eta u, \eta u)&= \int_B \langle a\nabla (\eta u), \nabla (\eta u) \rangle dx\\
        &\leq 2 \biggl(\int_B \langle a\nabla u, \nabla u \rangle \eta^2 dx  + \int_B \langle a \nabla\eta, \nabla\eta \rangle u^2 dx \biggr)\\
        &\leq 2\CE_\eta(u,u) + 2\int_B\Lambda|\nabla\eta|^2u^2dx\\
        &\leq 2\CE_\eta(u,u) + 2\Lambda\norm{\nabla\eta}_\infty^2\norm{u}_{L^2(B)}^2.
    \end{align*}
    Combining this and $(\ref{ine:inprop:cutoffSob})$, we get the desired inequality.
\end{proof}

For a function $\eta$, $x_0 \in \mathbb{R}^d$ and $R>0$, set
\begin{align*}
	(u)_{B_W(x_0,R)}^\eta := \left.\int_{B_W(x_0,R)}u(x)\eta(x)dx\middle/ \int_{B_W(x_0,R)}\eta(x)dx\right..
\end{align*}
We say that a cutoff function $\eta$ on $B_W(x_0,R)$ is a radial cutoff function if we have a representation $\eta(x) = \Phi(d_W(x_0, x)/R)$, where $\Phi \colon [0,\infty) \to [0,\infty)$ is a nonincreasing, nonnegative and c\`{a}dl\`{a}g function that is not identically zero on $(1/2, 1]$. We cite the following weighted Poincar\'{e} inequality due to Dyda and Kassmann (\cite{DK}).

\begin{lem}(\cite[Theorem 1]{DK})\label{lem:wPoi}
	Let $(X,\rho)$ be a metric space with a positive $ \sigma $-finite Borel measure $ \mu $. Let $ p \in [1,\infty) $ and let $ \phi $ be a function such that $ \phi = \Phi(\rho(\cdot,x_0)) $, where $\Phi \colon [0,\infty) \to [0,\infty)$ is a nonincreasing, nonnegative and c\`{a}dl\`{a}g function that is not identically zero on $(1/2, 1]$. We denote the open ball centered at $x_0$ with radius $r$ by $B_r = \{x\in X: \rho(x_0,x)<r \}$. Let $ F\colon L^p(X, \mu)\times (1/2,1] \to [0,\infty]$ be a functional satisfying
	\begin{align*}
	F(u+c,r) &= F(u,r), \quad u \in L^p(X, \mu),c \in \mathbb{R}, \\
	\norm{u - (u)_{B_r}}_p^p &\leq F(u,r), \quad u \in L^p(X, \mu),
	\end{align*}
	for every $ r \in [1/2,1) $. Then for $\displaystyle M = \frac{8^p\mu (B_1)\Phi(0)}{{\mu \bigl(B_{\frac{1}{2}}\bigr)\Phi(\frac{1}{2})}}  $,
	\begin{align}
	\int_{B_1}|u-(u)_{B_1}^\phi|^p\phi d\mu \leq M \int_{1/2}^1F(u,t)\nu(dt)
	\end{align}
	for every $ u \in L^p(B_1, \mu) $, where $ \nu $ is a $ \sigma $-finite Borel measure satisfying
	\begin{align*}
	\phi(x) = \int_{\rho(x,x_0)\vee 1/2}^1v(dt) = \int_{1/2}^{1}1_{B_t}(x)\nu(dt)
	\end{align*}
	for $ x \in B_1 \setminus \overline{B_{1/2}} $.
\end{lem}

For each $\alpha \in (0, \infty)$, nonempty open set $E \subset W$, and pair of functions $u,\phi$ on $E$ satisfying $\phi \geq 0$,  set $\norm{u}_{\alpha, E, \phi} = \biggl( \frac{1}{|E|} \int_E |u|^\alpha \phi \,dx\biggr)^\frac{1}{\alpha}$.

\begin{prop}[Poincar\'{e} inequality with radial cutoff]
	Suppose that Assumptions \ref{asm:volRegiso}--\ref{asm:Dir} hold. Let $R\geq \Rtheta$ and $B=B_W(x_0,r)$ with $x_0 \in B_W(0, R)$ and $r/2 \geq R^\frac{\theta}{d}$. Let $ \eta $ be a radial cutoff function on $B$ with $ \eta(x)  = \Phi(d_W(x_0,x)/r) $, where $\Phi$ is some non-increasing, non-negative c\`{a}dl\`{a}g function with compact support and which is not identically zero on $(1/2, 1]$. Then
		\begin{align}\label{ine:radialWeightedPoincare}
			\norm{u - (u)_{B}^{\eta^2}}_{2,B,\eta^2} \leq C_\textup{rad} |B|^{\frac{2-d}{d}}\CE_\eta(u,u), \;u\in \CF_\textup{loc},
		\end{align} 
		where $C_\textup{rad}= 2^{2d+4}C_\textup{P}C_\textup{V}^{\frac{2(2d-3)}{d}}\lambda^{-1}\frac{\Phi(0)}{\Phi(1/2)}$. 
\end{prop}
\begin{proof}
	Define a linear functional $F\colon L^2(W, dx)\times (r/2, r] \to [0,\infty]$ by $F(u, s) = C_\textup{P}C_\textup{V}^{-1}\lambda^{-1}|B_s|^\frac{2-d}{d}\CE_{B_s}(u,u)$ for $u \in \CF_\textup{loc}$ and $F(u, s)=\infty$ otherwise, where $B_s=B_W(x_0,s)$. Then the functional $F$ satisfies $F(u+c, s) = F(u,s)$ for all $c\in \mathbb{R}$. By Proposition \ref{prop:locPoi}, we have, for $u \in \CF_\textup{loc}$ and $s \in (r/2, r]$,
	\begin{align*}
		\norm{u - (u)_{B_s}}_{2,B_s}^2 \leq F(u,s).
	\end{align*}
	Hence, we can apply Lemma \ref{lem:wPoi}. Let $\nu$ be a measure satisfying 
	\begin{align*}
		\eta^2(x) = \int_{r/2}^r 1_{B_s}(x)\nu(ds).
	\end{align*}
	Using Assumption \ref{asm:volRegiso}, we have $|B_{r/2}|^\frac{2-d}{2} = \Bigl(\frac{|B_r|}{|B_{r/2}|}\frac{1}{|B_r|}\Bigr)^\frac{d-2}{d} \leq 2^{d-2}C_\textup{V}^\frac{2(d-2)}{d}|B_r|^\frac{2-d}{d}$. Then by definition of $F$ and Assumption \ref{asm:volRegiso} we have 
	\begin{align*}
		&\norm{u - (u)_{B_1}^{\eta^2}}_{2,B_1,\eta^2} 
		\leq M \int_{r/2}^r F(u,s)\nu(ds) \\
		&= \frac{2^6|B_r|\Phi(1)}{|B_{r/2}|\Phi(1/2)}\int_{r/2}^r \int_{B_r}C_\textup{P}C_\textup{V}^{-\frac{2}{d}}\lambda^{-1}|B_s|^\frac{2-d}{d} \CE_{B_s}(u,u)1_{B_s}dx\nu(ds) \\
		&\leq \frac{2^6|B_r|\Phi(1)}{|B_{r/2}|\Phi(1/2)}\int_{r/2}^r \int_{B_r}C_\textup{P}C_\textup{V}^{-\frac{2}{d}}\lambda^{-1} |B_{r/2}|^\frac{2-d}{d} \CE_{B_s}(u,u)1_{B_s}dx\nu(ds) \\
		&\leq 2^{d+6}C_\textup{V}^2\frac{\Phi(1)}{\Phi(1/2)}\int_{r/2}^r \int_{B_r}C_\textup{P}C_\textup{V}^{-\frac{2}{d}}\lambda^{-1} 2^{d-2}C_\textup{V}^\frac{2(d-2)}{d}|B_r|^\frac{2-d}{d} \CE_{B_s}(u,u)1_{B_s}dx\nu(ds) \\
		&\leq 2^{2d+4}C_\textup{P}C_\textup{V}^{\frac{2(2d-3)}{d}}\lambda^{-1}\frac{\Phi(1)}{\Phi(1/2)}\int_{r/2}^r \int_{B_r} |B_r|^\frac{2-d}{d}\langle a\nabla u, \nabla u  \rangle 1_{B_s} dx\nu(ds)\\
		&= C_\textup{rad} |B_r|^{\frac{2-d}{d}}\int_{B_r} \langle a\nabla u, \nabla u  \rangle  \eta^2 dx,
	\end{align*}
	which is the desired result.
\end{proof}

\subsection{Parabolic Harnack inequality and H\"{o}lder continuity}
\text{}

We denote the $L^2$-inner product by $(\cdot, \cdot )$.

\begin{Def}
    Let $ I \subset \mathbb{R} $ be an interval and $ G \subset W $ be an open set. We say that a function $ u \colon I \to \CF$ is a subcaloric (resp. supercaloric) function in $ I \times G $ if $ t \mapsto (u(t,\cdot),\phi)$ is differentiable in $ t \in I $ for any $ \phi \in L^2(G) $, and 
    \begin{align}
    \frac{d}{dt}(u(t,\cdot),\phi) + \CE(u_t,\phi) \leq 0, \;\; (\text{resp. }\geq)
    \end{align}
    for all non-negative $ \phi \in \CF_G$. We say that a function $ u \colon I \to \CF$ is a caloric function in $ I \times G $ if it is both subcaloric and supercaloric.
\end{Def}

We say that a function $\xi \in C_c^\infty(\mathbb{R})$ is a cutoff function in time if $0\leq \xi\leq 1$. 
The proof of the lemma below is the same as in \cite[Lemma B.3]{CD}.

\begin{lem}\label{lem:Fcal}
	Let $ F\colon \mathbb{R} \to \mathbb{R} $ be a twice differentiable function with bounded second derivative and positive first derivative. Assume that $ F'(0)=0 $. Then for any subcaloric (supercaloric) function $ u $ we have
	\begin{align*}
	\frac{d}{dt}(F(u_t),\phi)+\CE(u_t,F'(u_t)\phi)\leq 0 \;\; (\geq 0)
	\end{align*}   
	for all $ \phi \in C_c^\infty (\overline{W}) $, $ \phi >0 $ and $ t > 0 $.
\end{lem}

In the rest of this section, we assume that $\Lambda \geq 1$ by replacing $\Lambda$ with $\Lambda \vee 1$ if we need.
\begin{lem}\label{lem:nu-1spaceTime}
    Suppose that Assumptions \ref{asm:volRegiso} and \ref{asm:Dir} hold. Let $I = (t_1,t_2) \subset \mathbb{R}$, $ R \geq \Rtheta $ and $x_0 \in B_W(0, R)$. Let $B = B_W(x_0, r)$ with $r \geq R^\frac{\theta}{d}$. Let $ u $ be a locally bounded positive subcaloric function in $ Q = I \times B $. Take a  function $ \eta$  on $B$ and a cutoff function $\xi$ in time on $I$. Let $p \in (\frac{2d}{d+\zeta}, 2)$. Set $ \nu = 2(1 - \frac{d-p\zeta}{dp})$. Then for all $ \alpha \geq 1$
    \begin{align}\label{ine:spacetimecutoff}
        \norm{\xi \eta^2 u^{2\alpha}}_{\nu,I\times B}^\nu 
        \leq C_0\frac{|B|^\frac{2}{d}}{|I|^{1-\nu}}\Bigl(\alpha (\norm{\xi'}_\infty + \norm{\nabla\eta}_\infty^2) \Bigr)^\nu\norm{u^{2\alpha}}_{1,I\times B}^\nu,
    \end{align}
    where $C_0=2^{2\nu+1}\Lambda^{\nu - 1}(1+\Lambda)C_\textup{S}$.
\end{lem}
\begin{proof}
    By Lemma \ref{lem:Fcal}, we have
    \begin{align}\label{ine:inPropSpacetimecutoff1}
        \frac{d}{dt}(u_t^{2\alpha}, \eta^2) + 2\alpha\CE(u_t, u_t^{2\alpha-1}\eta^2) \leq 0.
    \end{align}
    By the Cauchy-Schwartz inequality, we have that 
    \begin{align*}
        &\CE(u_t, u_t^{2\alpha-1}\eta^2) \\
        &=  (2\alpha-1)\int_B\eta^2u_t^{2\alpha-2}\innerproduct{a\nabla u_t}{\nabla u_t}dx + 2\int_B\eta u_t^{2\alpha-1}\innerproduct{a\nabla u_t}{\nabla \eta}dx\\
        &= (2\alpha -1)\int_B \eta^2u_t^{2\alpha-2}\innerproduct{a\nabla u_t}{\nabla u_t}dx + \frac{2}{\alpha}\int_B \innerproduct{a(\eta \nabla (u_t^\alpha))}{u_t^{\alpha}\nabla\eta}dx\\
        &\geq \frac{2\alpha-1}{\alpha^2}\CE_\eta(u_t^\alpha, u_t^\alpha)- \frac{2}{\alpha}\Biggl(\int_B \innerproduct{a\nabla (u_t^\alpha)}{\nabla (u_t^\alpha)}\eta^2 dx \Biggr)^\frac{1}{2}\Biggl(\int_B \innerproduct{a\nabla \eta}{\nabla\eta}u_t^{2\alpha} dx\Biggr)^\frac{1}{2}.
    \end{align*} 
    Using Assumption \ref{asm:Dir}, the second term of the right-hand side is estimated below by
    \begin{align}\label{ine:inLemnu-1spaceTime1}
        &-\frac{2}{\alpha}\CE_\eta(u_t^\alpha, u_t^\alpha)^\frac{1}{2}\Biggl(\int_B |\nabla\eta|^2 u_t^{2\alpha} \Lambda dx\Biggr)^\frac{1}{2} \nonumber \\
        &\geq -\frac{2}{\alpha}\Lambda^\frac{1}{2}\norm{\nabla\eta}_\infty\CE_\eta(u_t^\alpha, u_t^\alpha)^\frac{1}{2}\norm{u_t^{2\alpha}}_{L^1(B)}^\frac{1}{2}.
    \end{align}
    Using the Young inequality $2ab \leq a^2/\epsilon + \epsilon b^2 $ with $a=\CE_\eta(u_t^\alpha, u_t^\alpha)^\frac{1}{2}$, $b=\norm{\nabla\eta}_\infty\norm{u_t^{2\alpha}}_{L^1(B)}^\frac{1}{2}$, and $\epsilon=2\alpha\Lambda^\frac{1}{2}$  we estimate 
    \begin{align}
        &-\frac{2}{\alpha}\Lambda^\frac{1}{2}\norm{\nabla\eta}_\infty\CE_\eta(u_t^\alpha, u_t^\alpha)^\frac{1}{2}\norm{u_t^{2\alpha}}_{L^1(B)}^\frac{1}{2}\nonumber \\
        &\geq -\frac{\CE_\eta(u_t^\alpha, u_t^\alpha)}{2\alpha^2} - 2\Lambda\norm{\nabla\eta}_\infty^2 \norm{u_t^{2\alpha}}_{L^1(B)}.
    \end{align} 
    Inserting this inequality into $(\ref{ine:inLemnu-1spaceTime1})$ and using the fact that $\alpha \geq 1$, we have
    \begin{align}
        &\CE(u_t, u_t^{2\alpha-1}\eta^2) \geq \frac{2\alpha-1}{\alpha^2}\CE_\eta(u_t^\alpha, u_t^\alpha) - \frac{\CE_\eta(u_t^\alpha, u_t^\alpha)}{2\alpha^2} - 2\Lambda\norm{\nabla\eta}_\infty^2 \norm{u_t^{2\alpha}}_{L^1(B)} \nonumber\\
        &= \frac{4\alpha-3}{2\alpha^2}\CE_\eta(u_t^\alpha, u_t^\alpha) - 2\Lambda\norm{\nabla\eta}_\infty^2 \norm{u_t^{2\alpha}}_{L^1(B)} \nonumber \\
        &= \frac{1}{2\alpha}(4-\frac{3}{\alpha})\CE_\eta(u_t^\alpha, u_t^\alpha) - 2\Lambda\norm{\nabla\eta}_\infty^2 \norm{u_t^{2\alpha}}_{L^1(B)}  \nonumber \\
        &\geq \frac{1}{2\alpha}\CE_\eta(u_t^\alpha, u_t^\alpha) - 2\Lambda\norm{\nabla\eta}_\infty^2 \norm{u_t^{2\alpha}}_{L^1(B)}.\label{ine:inPropSpacetimecutoff1last}
    \end{align}
    Combining $(\ref{ine:inPropSpacetimecutoff1})$ and $(\ref{ine:inPropSpacetimecutoff1last})$, we have 
    \begin{equation}
        \frac{d}{dt}\norm{(u_t^\alpha\eta)^2}_{L^1(B)} + \CE_\eta(u_t^\alpha, u_t^\alpha) \leq 4\alpha\Lambda\norm{\nabla\eta}_\infty^2 \norm{u_t^{2\alpha}}_{L^1(B)}.
    \end{equation}
    Since $\Lambda \geq 1$ and $|\eta| \leq 1$,  multiplying by $\xi$ and integrating over $(t_1, t)$, we get 
    \begin{align}
        &\xi(t) \norm{(u_t^\alpha\eta)^2}_{L^1(B)} + \int_{t_1}^t \xi(s)\CE_\eta(u_t^\alpha, u_t^\alpha)ds \nonumber \\
        &\leq4\alpha\Lambda(\norm{\xi'}_\infty +  \norm{\nabla\eta}_\infty^2)\int_{t_1}^t\norm{u_s^{2\alpha}}_{L^1(B)}ds.
    \end{align}
    Averaging in space and taking the supremum in time, we have 
    \begin{align}\label{ine:inPropSpacetimecutoff2}
        &\sup_{t \in I}\xi(t) \norm{(u_t^\alpha\eta)^2}_{1,B} + \int_I \xi(s)\frac{\CE_\eta(u_t^\alpha, u_t^\alpha)}{|B|}ds \nonumber \\
        &\leq 4\alpha\Lambda(\norm{\xi'}_\infty +  \norm{\nabla\eta}_\infty^2)\int_I\norm{u_s^{2\alpha}}_{1, B}ds.
    \end{align}
    Recall that $ \nu = 2(1 - \frac{d-p\zeta}{dp} ) > 1$. Set $\tilde{p}=\frac{dp}{d-p\zeta}$. Observe that $\nu^{-1}/\tilde{p} + (1-\nu^{-1})/2 = 1/(2\nu)$. By the generalized H\"{o}lder inequality $\norm{f}_{r,B} \leq \norm{f}_{p,B}^\beta\norm{f}_{q,B}^{1-\beta}$ with $p=\tilde{p}, q=2, r=2\nu, \beta = \nu^{-1}$, and $f = \eta u_t^{2\alpha}$ we have
    \begin{align}
        \norm{(\eta u_t^\alpha)^2}_{\nu,B}^\nu &= \norm{\eta u_t^\alpha}_{2\nu,B}^{2\nu} \nonumber \\
        &\leq \norm{\eta u_t^\alpha}_{\tilde{p}, B}^{2\nu\cdot\nu^{-1}}\norm{\eta u_t^\alpha}_{2,B}^{2\nu(1-\nu^{-1})}  \nonumber \\
        &=\norm{\eta u_t^\alpha}_{\tilde{p}, B}^2\norm{(\eta u_t^\alpha)^2}_{1,B}^{\nu-1}.
    \end{align}
    Integrating this inequality against $\xi(s)^\nu|I|^{-1}$ over $I$, we have 
    \begin{equation}\label{ine:inPropSpacetimecutoff2a}
        \frac{1}{|I|}\int_I \xi(s)^\nu\norm{\eta^2u_s^{2\alpha}}_{\nu, B}^\nu ds 
        \leq \frac{1}{|I|}\Biggl(\sup_{s\in I}\xi(s) \norm{(\eta u_s^\alpha)^2}_{1,B} \Biggr)^{\nu-1}\int_I\xi(s)\norm{\eta u_s^\alpha}_{\tilde{p}, B}^2ds.
    \end{equation}
    By $(\ref{ine:inPropSpacetimecutoff2})$ and the fact that $\nu < 2$, we estimate
    \begin{align}\label{ine:inPropSpacetimecutoff2b}
        &\Biggl(\sup_{s\in I}\xi(s) \norm{(\eta u_s^\alpha)^2}_{1,B} \Biggr)^{\nu-1} \nonumber\\
        &\leq (4\Lambda)^{\nu-1}\Biggl(\alpha(\norm{\xi'}_\infty +  \norm{\nabla\eta}_\infty^2)\int_I\norm{u_s^{2\alpha}}_{1, B}ds\Biggr)^{\nu-1}.
    \end{align}
    Using the cutoff Sobolev inequality $(\ref{ine:CutoffSobolev})$, we have 
    \begin{align*}
        &\int_I\xi(s)\norm{\eta u_s^\alpha}_{\tilde{p}, B}^2ds 
        \leq \int_I\xi(s) \cdot 2C_\textup{S} |B|^\frac{2}{d}\Biggl(\frac{
            \CE_\eta(u_s^\alpha,u_s^\alpha)}{|B|} + \Lambda\norm{\nabla\eta}_\infty^2 \norm{u_s^\alpha}_{2, B}^2 \Biggr)ds
    \end{align*}
    Using $(\ref{ine:inPropSpacetimecutoff2})$ again, we estimate
    \begin{align}\label{ine:inPropSpacetimecutoff2c}
        &\int_I\xi(s) 2C_\textup{S} |B|^\frac{2}{d}\Biggl(\frac{
            \CE_\eta(u^\alpha,u^\alpha)}{|B|} + \Lambda\norm{\nabla\eta}_\infty^2 \norm{u^\alpha}_{2, B}^2 \Biggr)ds\nonumber \\
        &\leq 10\alpha(1+\Lambda) C_\textup{S}|B|^\frac{2}{d}(\norm{\xi'}_\infty + \norm{\nabla\eta}_\infty^2)\int_I \norm{u_s^{2\alpha}}_{1,B}ds
    \end{align}
    Inserting $(\ref{ine:inPropSpacetimecutoff2b})$ and $(\ref{ine:inPropSpacetimecutoff2c})$ into $(\ref{ine:inPropSpacetimecutoff2a})$, we get the desired result.
\end{proof}

Set
\begin{align}	
	&Q =Q(\tau, x_0, s, r)  = (s-\tau r^2,s)\times B_W(x_0,r), \\
	&Q_\delta = (s-\delta\tau r^2,s)\times B_W(x_0,\delta r). \label{def:paraball}
\end{align}
Note that $Q_\delta$ is a subset of $Q$.

\begin{lem}\label{lem:supL2}
    Suppose that Assumptions \ref{asm:volRegiso} and \ref{asm:Dir} hold. Let $B = B_W(x_0, r)$, $p$, $\nu$ be as in Lemma \ref{lem:nu-1spaceTime}. Fix $ \tau > 0  $ and $ 1/2 \leq \sigma' < \sigma \leq 1 $. Let $ u_t $ be a positive locally bounded subcaloric function on $ Q = Q(\tau,s, r) $. Then there exists a positive constant $ C_1 = C_1(d,\nu,C_\textup{V}) $ such that 
    \begin{align}
        \sup_{Q_{\sigma'}}u \leq C_1 \tilde{C_0}^{\frac{1}{2\nu-2}}\tau^\frac{1}{2} \biggl( \frac{1+\tau^{-1}}{(\sigma - \sigma')^2} \biggr)^\frac{\nu}{2\nu-2}\norm{u}_{2,Q_\sigma},
    \end{align} 
    where $ \nu = 2 - 2p^*/\rho $ and $\tilde{C_0} = 2^{d+1}C_\textup{V}^\frac{2}{d} C_0$.
\end{lem}
\begin{proof}The proof is based on that in \cite[Theorem 3.2]{CD}. However, it is necessary to undertake some additional work in our setting.
    Set 
    \begin{align*}
        \sigma_k=\sigma' + 2^{-k}(\sigma-\sigma'), \quad \delta_k = 2^{-k-1}(\sigma-\sigma').
    \end{align*}
    Then we can easily show that $ \sigma_k - \sigma_{k+1} = \delta_{k} $. We show the inequality using Moser's iteration scheme. To do this, we need good cutoff functions. Consider a cutoff function in space $ \eta_k \in C_c^\infty(\overline{W}) $, such that $0\leq \eta_k \leq 1$, $\eta_k = 0$ on $\partial B_W(x_0, \sigma_k r)-\partial W$ and $ \eta_k \equiv 1 $ on $ B_W(x_0, \sigma_{k+1}r) $ and $ \norm{\nabla \eta_k}_\infty \leq 2/(r\delta_k) $. 
    Take also a cutoff function in a time $ \xi_k \colon \mathbb{R} \to [0,1] $, $ \xi_k \equiv 1 $ on $ I_{\sigma_{k+1}} = (s-\sigma_{k+1}\tau r^2, s), $ $ \xi_k \equiv 0 $ on $ (-\infty, s-\sigma_k\tau r^2) $ and $ \norm{\xi_k'}_\infty \leq 2/(r^2\tau\delta_k) $. Let $ \alpha_k = \nu^k  $. We estimate
    \begin{align*}
        \norm{u}_{2\alpha_{k+1}, Q_{\sigma_{k+1}}} = &\Biggl( \frac{1}{|Q_{\sigma_{k+1}}|}\int_{Q_{\sigma_{k+1}}} \xi_k\eta_k^2 u^{2\alpha_{k+1}} dxdt \Biggr)^\frac{1}{2\alpha_{k+1}} \\
        &\leq \Biggl( \frac{1}{|Q_{\sigma_{k+1}}|}\int_{Q_{\sigma_k}} \xi_k\eta_k^2u^{2\nu\alpha_k}  dxdt \Biggr)^\frac{1}{2\alpha_{k+1}} \\
        &=\Biggl( \frac{|Q_{\sigma_k}|}{|Q_{\sigma_{k+1}}|}\frac{1}{|Q_{\sigma_k}|}\int_{Q_{\sigma_k}} \xi_k\eta_k^2 u^{2\nu\alpha_k}  dxdt \Biggr)^\frac{1}{2\alpha_{k+1}}\\
        &= \biggl(\frac{|Q_{\sigma_k}|}{|Q_{\sigma_{k+1}}|}\biggr)^\frac{1}{2\alpha_{k+1}} \norm{\xi_k \eta_k^2 u^{2\alpha_k}}_{\nu, Q_{\sigma_k}}^\frac{\nu}{2\alpha_{k+1}} . 
        \end{align*}
    Using $(1)$ of Assumption \ref{asm:volRegiso}, we compute 
    \begin{align*}
        \frac{|Q_{\sigma_k}|}{|Q_{\sigma_{k+1}}|} = \frac{\sigma_kr^2}{\sigma_{k+1}r^2}\frac{|B_W(x,\sigma_{k}r )|}{|B_W(x,\sigma_{k+1}r)|} \leq \frac{\sigma_k}{\sigma_{k+1}} C_\textup{V}^2\biggl(\frac{\sigma_k}{\sigma_{k+1}}\biggr)^d.
    \end{align*} 
    Since $1/2 < \sigma_{k+1}<\sigma_k < 1$, we obtain 
    \begin{align*}
        \frac{|Q_{\sigma_k}|}{|Q_{\sigma_{k+1}}|} \leq 2^{d+1}C_\textup{V}^2.
    \end{align*}
    By $(1)$ of Assumption \ref{asm:volRegiso} and definition of $\xi_k$ and $\eta_k$, we also have 
    \begin{align*}
        \frac{|B|^\frac{2}{d}}{|I|^{1-\nu}}(\norm{\xi_k'}_\infty + \norm{\nabla_k}_\infty^2)^\nu
        &\leq \frac{C_\textup{V}^\frac{2}{d}\sigma_k^2r^2 }{(\sigma_k \tau r^2)^{1-\nu}} \biggl(\frac{2}{r^2 \tau \delta_k} + \frac{4}{r^2 \delta_k^2}\biggr)^\nu \\
        &=C_\textup{V}^\frac{2}{d}\tau^{\nu-1}\sigma_k^{\nu+1}r^{2\nu}
        \biggl(\frac{2\tau^{-1}}{r^2\delta_k} + \frac{4}{r^2\delta_k^2} \biggr)^\nu \\
        &\leq  C_\textup{V}^\frac{2}{d}\tau^{\nu-1}1^{\nu+1}r^{2\nu}
        \frac{1}{r^{2\nu}} \biggl(2^{2k}\frac{1+\tau^{-1}}{(\sigma -\sigma')^2}\biggr)^\nu\\
        &= C_\textup{V}^\frac{2}{d}\tau^{\nu-1}\biggl(2^{2k}\frac{1+\tau^{-1}}{(\sigma -\sigma')^2}\biggr)^\nu.
    \end{align*}
    Applying these estimations and $(\ref{ine:spacetimecutoff})$ with $ \alpha_{k+1}=\nu \alpha_k $, we have
    \begin{align*}
        \norm{u}_{2\alpha_{k+1}, Q_{\sigma_{k+1}}}  
        &\leq \Biggl(2^{d+1}C_\textup{V}^2C_0\tau^{\nu-1}\biggl[\frac{\alpha_k(1+\tau^{-1})2^{2k} }{(\sigma-\sigma')^2}\biggr]^\nu \Biggr)^\frac{1}{2\alpha_{k+1}} \norm{u}_{2\alpha_k, Q_{\sigma_k}}.
    \end{align*}
     Iterating the inequality from $ k=0 $ up to $ i $, we can find a constant $ C_1 > 0 $ which depends on $d,\; \nu$, and $ C_\textup{V} $ such that 
    \begin{align}\label{ine:itr}
        \norm{u}_{2\alpha_i, Q_{\sigma_i}}
        &\leq  \Biggl(2^{d+1}C_\textup{V}^2C_0\tau^{\nu-1}\biggl[\frac{\alpha_{i-1}(1+\tau^{-1})2^{2(i-1)} }{(\sigma-\sigma')^2}\biggr]^\nu \Biggr)^\frac{1}{2\alpha_{i}} \norm{u}_{2\alpha_{i-1}, Q_{\sigma_{i-1}}} \notag\\
        &\leq \Biggl(\prod_{k=0}^{i-1} \biggl(2^{d+1}C_\textup{V}^2C_0\tau^{\nu-1}\biggl[\frac{\alpha_k(1+\tau^{-1})2^{2k} }{(\sigma-\sigma')^2}\biggr]^\nu \biggr)^\frac{1}{2\alpha_{k+1}}  \Biggr) \cdot \norm{u}_{2, Q_{\sigma}} \notag\\
        &= (2^{d+1}C_\textup{V}^\frac{2}{d}C_0)^{\frac{1}{2}\sum_{k=0}^i \frac{1}{\alpha_k}}\tau^{\frac{\nu-1}{2}\sum_{k=0}^i\frac{1}{\alpha_k}} \notag \\
        &\times (4\nu)^{\frac{\nu}{2}\sum_{k=0}^i \frac{k}{\alpha_k}} \biggl(\frac{1+\tau^{-1}}{(\sigma-\sigma')^2} \biggr)^{\frac{\nu}{2}\sum_{k=0}^i \frac{1}{\alpha_k}} \norm{u}_{2,Q_\sigma}  
        \notag\\
        &\leq C_1(2^{d+1}C_\textup{V}^\frac{2}{d}C_0)^{\frac{1}{2\nu-2}} \tau^\frac{1}{2}\biggl(\frac{1+\tau^{-1}}{(\sigma - \sigma')^2} \biggr)^\frac{\nu}{2\nu-2} \norm{u}_{2, Q_{\sigma}},
    \end{align}
    where we used the fact $ \sum 1/\alpha_k = 1/(\nu-1) $ and that $\sum_{k=0}^\infty k/\alpha_k < \infty$ and $\nu \leq 2$. Increasing $ C_1 $ if needed, from $(\ref{ine:itr})$ we obtain
    \begin{align*}
        \norm{u}_{2\alpha_i, Q_{\sigma'}} \leq C_1\tilde{C_0}^{\frac{1}{2\nu-2}} \tau^\frac{1}{2}\biggl(\frac{1+\tau^{-1}}{(\sigma - \sigma')^2} \biggr)^\frac{\nu}{2\nu-2} \norm{u}_{2, Q_{\sigma}},
    \end{align*}
    Taking the limit as $ i \to \infty $ gives the result
    \begin{align*}
        \sup_{Q_{\sigma'}}u \leq C_1 \tilde{C_0}^{\frac{1}{2\nu-2}}\tau^\frac{1}{2} \biggl( \frac{1+\tau^{-1}}{(\sigma - \sigma')^2} \biggr)^\frac{\nu}{2\nu-2}\norm{u}_{2,Q_\sigma},
    \end{align*} 
\end{proof}

\begin{cor}\label{cor:subcalLL}
	Suppose that Assumptions \ref{asm:volRegiso} and \ref{asm:Dir} hold. Let $B = B_W(x_0, r)$, $p$, $\nu$ be as in Lemma \ref{lem:nu-1spaceTime}. Fix $ \tau > 0 $ and $ 1/2 \leq \sigma' < \sigma \leq 1 $. Let $ u $ be a subcaloric function in $ Q = Q(\tau, s, r) $. Then, there exists a positive constant $ C_2$, which depends only on $d,\nu,C_\textup{V}$, such that for all $ \alpha \geq 1 $
	\begin{align}
		\sup _{Q_{\sigma'}}u \leq C_2 2^{\frac{2}{\alpha^2}\frac{\nu}{\nu -1}}\tilde{C_0}^{\frac{1}{\alpha\nu-\alpha}}\tau^\frac{1}{\alpha}\biggl(\frac{1+\tau^{-1}}{(\sigma-\sigma')^2} \biggr)^\frac{\nu}{\alpha\nu - \alpha}\norm{u}_{\alpha,Q_\sigma}.
	\end{align}
\end{cor}

\begin{proof}
	When $ \alpha > 2 $, then it follows from Jensen's inequality.
	
	Let $ \alpha \in [1,2) $. Set $ \sigma_0 = \sigma'$ and $ \sigma_{i+1} = \sigma_i + \frac{\sqrt{2}-1}{\sqrt{2}}(\sigma - \sigma_i) $. Observe that $ \sigma - \sigma_i = (1/\sqrt{2})^i (\sigma - \sigma')$. Applying Lemma \ref{lem:supL2}, we have
	\begin{align}\label{ine:InsupL2}
		\sup _{Q_{\sigma_{i-1}}} u \leq C_1\tau^\frac{1}{2} \tilde{C_0}^{\frac{1}{2\nu-2}}\biggl(\frac{1+\tau^{-1}}{(\sigma-\sigma')^2}2^{i} \biggr)^\frac{\nu}{2\nu - 2} \norm{u}_{2,Q_{\sigma_i}}.
	\end{align}
	By (1) of Assumption \ref{asm:volRegiso}, we have
	\begin{align*}
		\norm{u}_{2,Q_{\sigma_i}} &=\Biggl(\frac{1}{|Q_{\sigma_i}|}\int_{Q_{\sigma_i}} u^{2(1-\frac{\alpha}{2})}u^\alpha dxdt \Biggr)^\frac{1}{2} \\
		&\leq \biggl(\sup_{Q_{\sigma_i}}u\biggr)^{1 - \frac{\alpha}{2}}\Biggl(\frac{1}{|Q_{\sigma_i}|}\int_{Q_{\sigma_i}} u^\alpha dxdt \Biggr)^\frac{1}{2}\\
		&\leq \biggl(\sup_{Q_{\sigma_i}}u\biggr)^{1 - \frac{\alpha}{2}}\Biggl(2^{d+1} C_\textup{V}^2 \frac{1}{|Q_\sigma|}\int_{Q_\sigma} u^\alpha dxdt \Biggr)^\frac{1}{2}\\
		&= 2^{\frac{d+1}{2}}C_\textup{V} \biggl(\sup_{Q_{\sigma_i}}u\biggr)^{1 - \frac{\alpha}{2}}\norm{u}_{\alpha,Q_\sigma}^\frac{\alpha}{2}.
	\end{align*}
	Inserting this into $(\ref{ine:InsupL2})$, we obtain 
	\begin{align*}
		\sup _{Q_{\sigma_{i-1}}} u \leq C_1 C_\textup{V} \tau^\frac{1}{2}2^{\frac{d+1}{2}} \tilde{C_0}^{\frac{1}{2\nu-2}}
		\biggl(\frac{1+\tau^{-1}}{(\sigma-\sigma')^2}2^i \biggr)^\frac{\nu}{2\nu - 2} 
		\biggl(\sup_{Q_{\sigma_i}}u\biggr)^{1 - \frac{\alpha}{2}}\norm{u}_{\alpha,Q_{\sigma_i}}^\frac{\alpha}{2}.
	\end{align*}
	Iterating this inequality we have
	\begin{align*}
		\sup _{Q_{\sigma_0}} u &\leq
		\Biggl( 
			C_1C_\textup{V} \tau^\frac{1}{2}2^\frac{d+1}{2} \tilde{C_0}^{\frac{1}{2\nu-2}}
			\biggl(\frac{1+\tau^{-1}}{(\sigma-\sigma')^2} \Biggr)^\frac{\nu}{2\nu - 2} \norm{u}_{\alpha,Q_{\sigma_i}}^\frac{\alpha}{2}
			\biggr) ^{\sum_{k=0}^i (1-\frac{\alpha}{2})^k } \\
			&\times  2^{\frac{1}{2}\frac{\nu}{\nu -1}\sum_{k=0}^i(k+1)(1-\frac{\alpha}{2})^k }
			\biggl(\sup_{Q_{\sigma_i}}u\biggr)^{(1 - \frac{\alpha}{2})^i}.
		\end{align*}
		Letting $  i \to \infty $, we get 
		\begin{align*}
			\sup _{Q_{\sigma'}}u \leq C_2 2^{\frac{\nu}{\nu -1}\frac{2}{\alpha^2}}\tilde{C_0}^{\frac{1}{\alpha\nu-\alpha}}\tau^\frac{1}{\alpha}\biggl(\frac{1+\tau^{-1}}{(\sigma-\sigma')^2} \biggr)^\frac{\nu}{\alpha\nu - \alpha}\norm{u}_{\alpha,Q_\sigma}.
		\end{align*}
		Thus the proof is over.
	\end{proof}

    \begin{lem}\label{lem:supercalInv}
        Suppose that Assumptions \ref{asm:volRegiso} and \ref{asm:Dir} hold. Let $B = B_W(x_0, r)$, $p$, $\nu$ be as in Lemma \ref{lem:nu-1spaceTime}. Fix $ \tau > 0 $ and let $ 1/2 < \sigma' \leq \sigma \leq 1 $. Let $ u_t $ be a positive supercaloric function on $ Q = Q(\tau, s,r) $.  Then there exists a positive constant $ C_1' $ which depends only on $ d, \nu, \Lambda, C_\textup{V}, C_\textup{I} $ such that for all $ \alpha >0 $
        \begin{align}
        \sup_{Q_{\sigma'}} u^{-\alpha} 
        \leq C_1' \tau\biggl(\frac{1+\tau^{-1}}{(\sigma-\sigma')^2} \biggr)^\frac{\nu}{\nu-1}\norm{u^{-1}}_{\alpha,Q_{\sigma}}^\alpha.
        \end{align} 
    \end{lem}
    
    \begin{proof}
        We can always assume that $ u > \epsilon $ by considering the supercaloric function $ u + \epsilon $ and then sending $ \epsilon $ to zero at the end of the argument. Applying Lemma \ref{lem:Fcal} with a smooth function $ F(x)$ such that $F(x) = -|x|^{-\beta} $ for $x \geq \epsilon$ and $ \beta > 0 $ we get
        \begin{align}\label{ine:InsupercalInv1}
        -\frac{d}{dt}\norm{\eta^2u_t^{-\beta}}_{1,B} + \beta\CE(u_t^{-\beta-1}\eta^2, u_t) \geq 0.
        \end{align}
        The second term of the left-hand side is equal to 
        \begin{align*}
            &\beta\int_B \langle a \nabla(u_t^{-\beta - 1}\eta^2), \nabla u_t \rangle dx \\
            &= -\beta(\beta +1) \int_B \langle a\nabla u_t, \nabla u_t  \rangle \eta^2 u_t^{-\beta-2} dx 
            + 2\beta\int_B \langle a\nabla \eta, \nabla u_t \rangle \eta u_t^{-\beta -1} dx \\
            &= -4 \frac{\beta+1}{\beta}\int_B \langle a \nabla(u_t^{-\beta/2}), \nabla(u_t^{-\beta/2})  \rangle \eta^2 dx 
            -4 \int_B \langle a \nabla \eta, \nabla (u_t^{-\beta/2}) \rangle \eta u_t^{-\beta/2} dx\\
            &= -4 \frac{\beta+1}{\beta}\CE_\eta(u_t^{-\beta/2}, u_t^{-\beta/2}) - 4\int_B \langle a(u_t^{-\beta/2}\nabla \eta), \eta\nabla(u_t^{-\beta/2}) \rangle dx. 
        \end{align*}
        Inserting this into $(\ref{ine:InsupercalInv1})$ we have
        \begin{align*}
        -\frac{d}{dt}\norm{\eta^2u_t^{-\beta}}_{L^1(B)} -4 \frac{\beta +1}{\beta} \CE_\eta(u_t^{-\beta/2}, u_t^{-\beta/2})-4 \int_B\langle a(u_t^{-\beta/2}\nabla\eta), \eta\nabla(u_t^{-\beta/2}) \rangle dx \geq 0.
        \end{align*}
        By means of Young's inequality $ 2ab\leq  a^2/3 + 3b^2  $ and using the fact that $ (\beta+1)/\beta  > 1$, we get, after averaging
        \begin{align}
        \frac{d}{dt}\norm{\eta^2u_t^{-\beta}}_{1,B} + \frac{\CE_\eta(u_t^{-\beta/2}, u_t^{-\beta/2} )}{|B|}
        \leq c\norm{\nabla\eta}_\infty^2\norm{u_t^{-\beta}}_{1,B} .
        \end{align}
        We now integrate against a time cutoff function $ \xi \colon \mathbb{R} \to [0,1] $. The same approach as in Lemma \ref{lem:nu-1spaceTime} applies and we get
        \begin{align*}
        \norm{\xi \eta^2 u_t^{-\beta}}_{\nu, I\times B}^\nu \leq c  \frac{|B|^\frac{2}{d}}{|I|^{1-\nu}} \bigl(\norm{\xi'}_\infty + \norm{\nabla \eta}_\infty^2\bigr)^\nu \norm{u^{-\beta}}_{1,I\times B}.
        \end{align*}
        Moser's iteration scheme with $ \beta_k = \nu^k\alpha $ and $ \alpha > 0 $ the same argument of Theorem \ref{lem:supL2} will give
        \begin{align*}
        \sup_{Q_{\sigma'}} u^{-\alpha} \leq C_1'\tau\biggl(\frac{1+\tau^{-1}}{(\sigma-\sigma')^2} \biggr)^\frac{\nu}{\nu-1}\norm{u^{-1}}_{\alpha,Q_{\sigma}}^\alpha,
        \end{align*}
        which is the desired result.
    \end{proof}

    Set 
    \begin{align}
        Q_{\delta}' =  Q_{\delta}'(\tau, x_0, s, r)= (s-\tau r^2, s-(1-\delta)\tau r^2)\times B_W(x_0, \delta r).
    \end{align}

    \begin{lem}\label{lem:supercalLL}
        Suppose that Assumptions \ref{asm:volRegiso} and \ref{asm:Dir} hold. Let $B = B_W(x_0, r)$ be as in Lemma \ref{lem:nu-1spaceTime}. Fix $ \tau > 0 $. Let $ 1/2 \leq \sigma' < \sigma \leq 1 $. Let $ u_t $ be a positive supercaloric function in $ Q = Q(\tau,s,r) $. Fix $ 0 < \alpha_0 <\nu $. Then there exists a positive constant $ C_3 $ which depends only on $ d, \nu, C_\textup{V} $ and on $ \alpha_0 $ such that for all $ 0 < \alpha < \alpha_0\nu^{-1} $ we have 
        \begin{align}
        \norm{u}_{\alpha_0, Q_{\sigma'}'} \leq 
        \Biggl(C_3 \tau (1+\tau^{-1})^\frac{\nu}{\nu-1}\biggl( \frac{1}{(\sigma-\sigma')^2}  \biggr)^\frac{\nu}{\nu-1}  \Biggr)^{(1+\nu)(1/\alpha-1/\alpha_0) } \norm{u}_{\alpha,Q_\sigma'}.
        \end{align}
    \end{lem}
    \begin{proof}
        Assume that $ u $ is a supercaloric on $ Q = I \times B $. We can assume that $u \geq \epsilon$ as in Lemma \ref{lem:supercalInv}. Applying Lemma \ref{lem:Fcal} with the function $ F(x) = |x|^\beta $ with $ \beta \in (0,1) $ we get
        \begin{align}\label{ine:InsupercalLL1}
        \frac{d}{dt}\norm{\eta^2u_t^{\beta}}_{L^1(B)} + \beta\CE(u_t^{\beta-1}\eta^2, u_t) \geq 0.
        \end{align}
        The second term of the left-hand side is equal to
        \begin{align*}
            &\beta \int_B \langle a\nabla (u_t^{\beta-1}\eta^2), \nabla u_t  \rangle dx \\
            &= \beta(\beta-1) \int_B \langle a\nabla u_t, \nabla u_t \rangle \eta^2 u_t^{\beta-2} dx 
            +2\beta \int_B \langle a\nabla\eta, \nabla u_t \rangle \eta u_t^{\beta-1} dx\\
            &= 4 \frac{\beta-1}{\beta}\int_B \langle a\nabla (u_t^{\beta/2}), \nabla(u_t^{\beta/2})  \rangle\eta^2 dx 
            +4 \int_B \langle a\nabla\eta, \nabla(u_t^{\beta/2}) \rangle \eta u_t^{\beta/2} dx \\
            &= 4 \frac{\beta-1}{\beta}\CE_\eta(u_t^{\beta/2} , u_t^{\beta/2}) + 4 \int_B \langle a\nabla\eta, \nabla(u_t^{\beta/2})\rangle \eta u_t^{\beta/2} dx.
        \end{align*}
        Inserting this into $(\ref{ine:InsupercalLL1})$, we obtain
        \begin{align*}
        \frac{d}{dt}\norm{\eta^2u_t^{\beta}}_{L^1(B)} + 4 \frac{\beta - 1}{\beta}\CE_\eta(u_t^{\beta/2},u_t^{\beta/2}) + 4\int_B \langle a\nabla \eta, \nabla (u_t^{\beta/2}) \rangle \eta u_t^{\beta/2} dx \geq 0 .
        \end{align*} 
        Note that $ (\beta -1) $ is negative. If we take $ 0 < \beta < \alpha_0\nu^{-1} $ then we have
        \begin{align*}
        \frac{1 - \beta}{\beta} > 1 - \beta > 1 - \alpha_0/\nu =: \epsilon,
        \end{align*}
        this yields after Young's inequality
        \begin{align*}
        - \frac{d}{dt}\norm{\eta^2u_t^\beta}_{L^1(B)} + \epsilon \CE_\eta(u_t^{\beta/2},u_t^{\beta/2}) \leq A \norm{\nabla\eta}_\infty^2\norm{u_t^\beta}_{L^1(B)},
        \end{align*}
        where $ A $ is a constant possibly depending on $d, \nu, C_\textup{V}, \alpha_0 $ which will be changing throughout the proof. Now we introduce a time cutoff function $ \xi \colon \mathbb{R} \to [0,1] $, $ \xi \equiv 0 $ on $(t_2,\infty),$ where $ I = (t_1,t_2) $. This gives after integrating,
        \begin{align*}
        \xi(t)\norm{\eta^2u_t^\beta}_{L^1(B)} + \int_{t_1}^{t_2}\xi(s)\CE_\eta(u_s^{\beta/2},u_s^{\beta/2})ds
        \leq A \bigl(\norm{\xi'}_\infty + \norm{\nabla\eta}_\infty^2\bigr) \int_{t_2}^{t_2}\norm{u_s^\beta}_{L^1(B)}ds.
        \end{align*}
        Starting from this inequality, and repeating the argument we used for subcaloric functions, we end up with
        \begin{align}\label{ine:super:nu-1spcetimecutoff}
        \norm{\xi\eta^2u_t^\beta}_{\nu,I \times B}^\nu \leq A \frac{|B|^\frac{2}{d}}{|I|^{1-\nu}}\bigl(\norm{\xi'}_\infty + \norm{\nabla\eta}_\infty^2 \bigr)^\nu\norm{u^\beta}_{1,I\times B}^\nu.	
        \end{align}
        The idea is now to iterate $(\ref{ine:super:nu-1spcetimecutoff})$ with an appropriate choice of exponents parabolic balls and cutoff functions. 
        
        We define the exponents $ \alpha_i := \alpha_0\nu^{-i} $ and $ \beta_j = \alpha_i\nu^{j-1} $ for $ j = 1, \dots i $. Observe that $ 0 < \beta_j < \alpha_0\nu^{-1} $ and thus we are in a setting where $(\ref{ine:super:nu-1spcetimecutoff})$ is applicable.
        
        We define the parabolic balls. We fix $ \sigma_0 = \sigma $, $ \sigma_j - \sigma_{j+1} = \delta_j := 2^{-j-1}(\sigma-\sigma') $, and set for $ j = 1,\dots, i $
        \begin{align*}
        I_{\sigma_j} = (s - \tau r^2, s- (1-\sigma_j)\tau r^2), \quad Q_{\sigma_j}' = I_{\sigma_j}\times B_W(x_0, \sigma_j r).
        \end{align*}
        
        For $ j = 1, \dots ,i $ we define the cutoff functions $ \eta_j \colon \mathbb{R}^d \to [0,1] $, such that $ \text{supp}\;\eta_j \subset B_W(x_0, \sigma_jr) $, $ \eta_j \equiv 1 $ on $ B_W(x_0, \sigma_{j+1}r) $ and $ \norm{\nabla\eta_j}_\infty \leq 2/(r\delta_j)$, and the  functions $ \xi_j\colon \mathbb{R} \to [0,1] $, $ \xi_j \equiv 1 $ on $ I_{\sigma_{j+1}} $, $ \xi_j \equiv 0 $ on $ (s-(1-\delta_j)\tau r^2 ,\infty) $ and $ \norm{\xi_j'}_\infty \leq 2/(r^2\tau\delta_j) $. Then as in the proof of Lemma \ref{lem:supL2}, we have 
        \begin{align*}
            \frac{|B|^\frac{2}{d}}{|I|^{1-\nu}}(\norm{\xi_k'}_\infty + \norm{\nabla \eta_k}_\infty^2)^\nu
            &\leq C_\textup{V}^\frac{2}{d}\tau^{\nu-1}\biggl(2^{2k}\frac{1+\tau^{-1}}{(\sigma -\sigma')^2}\biggr)^\nu.
        \end{align*}
        We are ready to apply $(\ref{ine:super:nu-1spcetimecutoff})$ for $ j = 1,\dots,i $ and the choices above:
        \begin{align*}
        \norm{u_t^{\alpha_i\nu^j}}_{\nu,Q_{\sigma_j}'}^\nu
        &= \norm{\xi_j \eta_j^2 u_t^{\alpha_i\nu^j}}_{\nu,Q_{\sigma_j}'}^\nu\\
        &\leq A\norm{\xi_j \eta_j^2 u_t^{\alpha_i\nu^j}}_{\nu,Q_{\sigma_j-1}'}^\nu\\
        &\leq A \tau^{\nu-1}
        \biggl(\frac{(1+\tau^{-1})2^{2j}}{(\sigma-\sigma')^2} \biggr)^\nu  
        \norm{u^{\alpha_i\nu^{j-1}}}_{1,Q_{\sigma_{j-1}}'}^\nu
        \end{align*}
        which after iteration from $ j=1 $ to $ i $ gives
        \begin{align*}
        \norm{u}_{\alpha_0,Q_{\sigma_i}'}^{\alpha_0} \leq
        2^{2\sum_{j=0}^{i-1}(i-j)\nu^j }\Biggl(A\tau^{\nu -1}\biggl(\frac{1+\tau^{-1}}{(\sigma-\sigma')^2} \biggr) \Biggr)^{\sum_{j=0}^{i-1}\nu^j}\norm{u^{\alpha_i}}_{1,Q_{\sigma}'}^{\nu^i}.
        \end{align*} 
        Now observe that
        \begin{align}
        &\sum_{j=0}^{i-1} (i-j)\nu^j \leq \frac{1}{2-(\nu+\nu^{-1})}\biggl(\frac{\alpha}{\alpha_0}-1\biggr) =: C(\nu)\biggl(\frac{\alpha}{\alpha_0}-1\biggr),\\
        &\sum_{j=0}^{i-1}\nu^j = \frac{\nu^i-1}{\nu-1} = \frac{\alpha/\alpha_i-1}{\nu-1}
        \end{align}
        This yields the following inequality
        \begin{align}
        \norm{u}_{\alpha_0,Q_{\sigma'}'} \leq \Biggl(A\tau (1+\tau^{-1})^\frac{\nu}{\nu-1} \biggl(\frac{1}{(\sigma-\sigma')^2} \biggr) \Biggr)^{1/\alpha_i-1/\alpha_0 }\norm{u}_{\alpha_i,Q_\sigma'}.
        \end{align}
        Replacing $ A $ by $ A \vee 1 $, we can assume that $ A $ is greater than one. Finally we extend the inequality for $ \alpha \in (0,\alpha_0\nu^i) $. Let $ i\geq 2 $ be an integer such that $ \alpha_i \leq \alpha \leq \alpha_{i-1} $, then $ 1/\alpha_i - 1/\alpha_0 \leq (1+\nu)(1/\alpha-1/\alpha_0) $ and by means of Jensen's inequality we obtain the desired result.
    \end{proof}
    \begin{lem}\label{lem:CalLog}
        Suppose that Assumptions \ref{asm:volRegiso} and \ref{asm:Dir} hold. Let $B = B_W(x_0, r)$, $p$, $\nu$ be as in Lemma \ref{lem:nu-1spaceTime}. Fix $ \tau > 0 $ and $ \kappa \in (0,1), \; \delta \in [1/2,1) $. For any $ s \in \mathbb{R} $  and any positive supercaloric function $ u $ on $ Q = (s-\kappa\tau r^2, s)\times B$, there exist a positive constant $ C_4 = C_4(d, \lambda, \Lambda, \nu, C_\textup{V}, C_\textup{P}, \delta) $ and a constant $ k := k(u,\kappa)  > 0$ such that 
        \begin{align}\label{ine:CalLog+}
        \mu(\{(t,x) \in K^+ \mid \log u_t < -l-k \}) \leq C_4(1\vee \tau^2)|B|^\frac{2+d}{d}l^{-1}
        \end{align} 
        and
        \begin{align}\label{ine:CalLog-}
        \mu(\{(t,x) \in K^- \mid \log u_t > l-k \})
        \leq C_4 (1 \vee \tau^2) |B|^\frac{2+d}{d}l^{-1}
        \end{align} 
        where $\mu = dt\otimes dx$, $ K^+ = (s-\kappa\tau r^2,s)\times B_W(x_0, \delta r) $, and $ K^- = (s-\tau r^2, s-\kappa\tau r^2)\times B_W(x_0,\delta r) $.
    \end{lem}
    \begin{proof}
        We can always assume $ u_t \geq \epsilon $ and then send $ \epsilon $ to zero in our estimates, since $ u_t+\epsilon $ is still a supercaloric function. By Lemma \ref{lem:Fcal},
        \begin{align}\label{ine:InLogEstbyFcal}
        \frac{d}{dt}(\eta^2,-\log u_t) &\leq \CE(u_t^{-1}\eta^2,u_t)
        = \int_B \langle a\nabla(u_t^{-1}\eta^2), u_t \rangle dx \notag \\ 
        &= -\int_B \langle a\nabla u_t, \nabla u_t \rangle \eta^2 u_t^{-2}dx +  2 \int_B\langle a\nabla \eta,\nabla u_t \rangle\eta u_t^{-1}dx \notag \\
        &= -\int_B \langle a\nabla \log u_t, \nabla \log u_t \rangle \eta^2 dx +  2 \int_B\langle a\nabla \eta,\nabla u_t \rangle\eta u_t^{-1}dx \notag \\
        &= -\CE_\eta(\log u_t,\log u_t) + 2 \int_B\langle a\nabla \eta,\nabla u_t \rangle\eta u_t^{-1}dx \notag \\
        &\leq -\CE_\eta(\log u_t,\log u_t) + 2\Lambda^\frac{1}{2}|B|^\frac{1}{2}\CE_\eta(\log u_t, \log u_t)^{1/2}\norm{\nabla\eta}_\infty \notag \\
        &\leq -\frac{1}{2}\CE_\eta(\log u_t,\log u_t) + 2\Lambda|B|\norm{\nabla\eta}_\infty,
        \end{align}
        where in the last inequality we used Young's inequality $ 2ab \leq 2\epsilon a^2+b^2/(2\epsilon) $ with $a=\norm{\nabla\eta}_\infty$, $b=\CE_\eta(\log u_t,\log u_t)^\frac{1}{2}$, and $\epsilon = \Lambda^{1/2}|B|^\frac{1}{2}$. We take radial cutoff function$ \eta $ by 
        \begin{align*}
        \eta(x) = (1 - d_W(x_0, x) /r)_+.
        \end{align*} 
        Set 
        \begin{align*}
        v_t(x) := -\log u_t(x), \quad V_t := (v_t)_{B}^{\eta^2}.
        \end{align*}
        Then the radial weighted Poincar\'{e} inequality $(\ref{ine:radialWeightedPoincare})$ reads
        \begin{align*}
            \norm{v_t-V_t}_{2,B,\eta^2}^2 \leq C_\textup{rad} |B|^\frac{2-d}{d}\CE_\eta(v_t,v_t),
            \end{align*}
        which implies 
        \begin{align}
        \frac{|B|}{\norm{\eta^2}_{L^1(B)}}\norm{v_t-V_t}_{2,B,\eta^2}^2 \leq C_\textup{rad} |B|^\frac{2}{d}\frac{\CE_\eta(v_t,v_t)}{\norm{\eta^2}_{L^1(B)}}.
        \end{align}
        Inserting this into $(\ref{ine:InLogEstbyFcal})$ we get
        \begin{align}\label{ine:inLemCalLogAfterRadialWeightedPoincare}
        \partial_t V_t + \frac{|B|}{\norm{\eta^2}_{L^1(B)}}
        \Bigl(2C_\textup{rad}|B|^\frac{2}{d} \Bigr)^{-1}\norm{v_t-V_t}_{2,B,\eta^2}^2 \leq 2\Lambda \norm{\nabla\eta}_\infty^2\frac{|B|}{\norm{\eta^2}_{L^1(B)}}.
        \end{align}
        Since $|B| \geq 1$ and $0\leq \eta \leq 1$, we have 
        \begin{align}\label{ine:inLemCalLogPreparation1}
            |B|/\norm{\eta^2}_{L^1(B)}\geq 1/\norm{1}_{L^1(B)}= |B|^{-1}.    
        \end{align}
         We will denote the concentric ball of radius $\delta r$ of $B$ by $\delta B$. Since $\eta \geq 1-\delta$ on $\delta B$, we have 
        \begin{align}\label{ine:inLemCalLogPreparation2}
            \frac{|B|}{\norm{\eta}_{L^1(B)}} \leq \frac{|B|}{\int_{\delta B}\eta^2dx}\leq \frac{|B|}{(1-\delta)^2|\delta B|} \leq \frac{C_\textup{V}^2C_\textup{rad}}{(1-\delta)^2\delta^d},
        \end{align}
        where we used Assumption \ref{asm:volRegiso}  in the last inequality. We also estimate
        \begin{align}\label{ine:inLemCalLogPreparation3}
            \norm{v_t-V_t}_{2,B,\eta^2} \geq (1-\delta)^2\int_{B_W(x_0, \delta r)}|v_t - V_t|^2dx.
        \end{align}
        By definition of $\eta$ and Assumption \ref{asm:volRegiso}, there exists a constant $c' > 0$ depending on $C_\textup{V}$ such that $\eta$ satisfies 
        \begin{align}\label{ine:inLemCalLogEstmationMeasure4}
            \norm{\nabla \eta}_\infty^2 \leq c' |B|^{-\frac{2}{d}}.
        \end{align}
        Inserting $(\ref{ine:inLemCalLogPreparation1}), (\ref{ine:inLemCalLogPreparation2})$,  $(\ref{ine:inLemCalLogPreparation3})$, and $(\ref{ine:inLemCalLogEstmationMeasure4})$ into $(\ref{ine:inLemCalLogAfterRadialWeightedPoincare})$ we obtain
        \begin{align}\label{ine:InLogAfter_wrPoi}
        \partial_t V_t + \Bigl(2(1-\delta)^{-2}C_\textup{rad}|B|^\frac{2+d}{d} \Bigr)^{-1} \int_{B_W(x_0, \delta r)}|v_t-V_t|^2  dx
        \leq c C_\textup{rad}|B|^{-\frac{2}{d}}
        \end{align}
        where  $ c $ is a constant depending only on $ d $, $\lambda$, $\Lambda$, $ \delta $, and $C_\textup{V}$. Next we introduce the following functions:
        \begin{align*}
        \tilde{v}_t := v_t - c C_\textup{rad}|B|^{-\frac{2}{d}}(t-s'), \quad \tilde{V}_t := V_t -c C_\textup{rad}|B|^{-\frac{2}{d}}(t-s'),
        \end{align*} 
        where $ s' = s - \kappa \tau r^2 $. We can rewrite $(\ref{ine:InLogAfter_wrPoi})$ as
        \begin{align}\label{ine:Rewrite_InLogAfter_wrPoi}
        \partial_t \tilde{V}_t + \Bigl(2(1-\delta)^{-2}C_\textup{rad}|B|^{\frac{2+d}{d}}\Bigr)^{-1} \int_{B_W(x_0, \delta r)}|\tilde{v_t}-\tilde{V_t}|^2 dx \leq 0.
        \end{align}
        Now set $ k(u,\kappa):= \tilde{V}_{s'} $, and define the two sets
        \begin{align*}
        D_t^+(l) := \{x \in B_W(x_0, \delta r) \mid \tilde{v}_t > k+l \}, \;\;\;
        D_t^-(l) := \{x \in B_W(x_0,\delta r) \mid \tilde{v}_t < k-l \}.
        \end{align*}
        Since $ \partial \tilde{V}_t \leq 0 $, for $ t > s' $, we have that $\tilde{v}_t-\tilde{V}_t > l+k-\tilde{V}_t \geq l $ on $ D_t^+(l) $. Using this and $(\ref{ine:Rewrite_InLogAfter_wrPoi})$ we have
        \begin{align*}
        \partial_t \tilde{V}_t + \Bigl(2(1-\delta)^{-2}C_\textup{rad}|B|^\frac{2+d}{d} \Bigr)^{-1}(l+k-\tilde{V}_t)^2 |D_t^+(l)| \leq 0 ,
        \end{align*}
        or equivalently
        \begin{align}\label{ine:inLemCalLogAfterSomeManipulation}
        -2(1-\delta)^{-2}C_\textup{rad}|B|^\frac{2+d}{d}\partial_t |l+k-\tilde{V}_t|^{-1} \geq |D_t^+(l)|.
        \end{align} 
        Integrating from $ s' $ to $ s $ yields
        \begin{align*}
        \mu(\{(t,x) \in K^+ \mid \tilde{v}_t(x) > k+l \}) \leq 2(1-\delta)^{-2} C_\textup{rad}|B|^\frac{2+d}{d}l^{-1}.
        \end{align*}
        Recall that $ -\log u_t = \tilde{v_t} + cC_\textup{rad}|B|^{-\frac{2}{d}} (t-s') $. We have
        \begin{align}\label{ine:inLemCalLogEstmationMeasure1}
        &\mu\biggl (\Bigl \{(t,x) \in K^+ \mid \log u_t(x) + cC_\textup{rad}|B|^{-\frac{2}{d}}(t-s') < -k-l \Bigr \}\biggr ) \nonumber \\
        &\leq 2(1-\delta)^{-2}C_\textup{rad}|B|^\frac{2+d}{d}l^{-1}.
        \end{align}
        Now we estimate 
        \begin{align}\label{ine:inLemCalLogEstmationMeasure2}
        &\mu\bigl(\{(t,x) \in K^+ \mid \log u_t < -k-l\} \bigr) \nonumber \\
        &\leq \mu \bigl( \{(t,x) \in K^+ \mid \log u_t(x) + cC_\textup{rad}|B|^{-\frac{2}{d}} (t-s') < -k-l/2  \} \bigr)\nonumber \\
        &+ \mu \bigl( \{(t,x) \in K^+ \mid  cC_\textup{rad}|B|^{-\frac{2}{d}}(t-s') > l/2  \}\bigr).
        \end{align}
        By Markov's inequality, we have 
        \begin{align}\label{ine:inLemCalLogEstmationMeasure3}
            &\mu \bigl( \{(t,x) \in K^+ \mid  cC_\textup{rad}|B|^{-\frac{2}{d}}(t-s') > l/2  \}\bigr)  \leq \frac{2}{l}\int_{K^+}cC_\textup{rad}|B|^{-\frac{2}{d}}(t-s')d\mu \nonumber\\
            &= l^{-1}cC_\textup{rad}|B|^{-\frac{2}{d}}(\kappa\tau r^2)^2 |B_W(x_0, \delta r)|  \leq cC_\textup{rad}C_\textup{V}^\frac{4}{d}|B|^\frac{2+d}{d}\tau^2 l^{-1}, 
        \end{align} 
         where we used the fact that $ \kappa <1 $ and Assumption \ref{asm:volRegiso} in the last inequality.
        Set $\tilde{c} = 4(1-\delta)^{-2}\vee cC_\textup{V}^\frac{4}{d}$. A combination of $(\ref{ine:inLemCalLogEstmationMeasure1})$, $(\ref{ine:inLemCalLogEstmationMeasure2})$, and $(\ref{ine:inLemCalLogEstmationMeasure3}$) yields
        \begin{align*}
            &\mu\bigl(\{(t,x) \in K^+ \mid \log u_t < -k-l\} \bigr) \\
            &\leq 4(1-\delta)^{-2}C_\textup{rad}|B|^\frac{2+d}{d}l^{-1} + cC_\textup{rad}C_\textup{V}^\frac{4}{d}|B|^\frac{2+d}{d}\tau l^{-1} \\
            &\leq 2\tilde{c}C_\textup{rad}(1\vee\tau^2)|B|^\frac{2+d}{d}l^{-1},
        \end{align*}
        which is the desired inequality $(\ref{ine:CalLog+})$. A similar argument shows the $(\ref{ine:CalLog-})$.
    \end{proof}

Before taking our next step, we cite Bombieri-Giusti's lemma. Intuitively, this lemma tells us that if we have some norm comparison inequality and an inequality for $\log u$, we have a uniform bound for the averaged $L^{\alpha_0}$-norm and the $L^\infty$-norm. See \cite{BG} for the details.

\begin{lem}\label{lem:Bombieri-Giusti}
	Let $ (X,\mathcal{X},\mu) $ be a measure space. Consider a collection of measurable subsets with finite measure $ \{U_\sigma\}_{\sigma \in (0,1]} $ such that $ U_{\sigma'} \subset U_\sigma $  whenever $ \sigma' \leq \sigma  $. Fix $ \delta \in (0,1) $. Let $ \kappa $ and $ K_1, \; K_2 $ be positive constants and $ 0 < \alpha_0 \leq \infty $. Let $ u $ be a positive measurable function on $ U := U_1 $ which satisfies 
	\begin{align}
	\Biggl( \int_{U_{\sigma'}}|u|^{\alpha_0}d\mu \Biggr)^\frac{1}{\alpha_0} \leq \bigl(K_1(\sigma-\sigma')^{-\kappa}\mu(U)^{-1} \bigr)^{\frac{1}{\alpha}-\frac{1}{\alpha_0}}\Biggl(\int_{U_\sigma}|u|^\alpha d\mu\Biggr)^\frac{1}{\alpha}
	\end{align}
	for all $ \sigma, \; \sigma' $ and $ \alpha $ such that $ 0 < \delta \leq \sigma' < \sigma \leq 1  $ and $ 0 < \alpha \leq \min\{1,\alpha_0/2\} $. Assume further that $ u $ satisfies 
	\begin{align}
	\mu(\log u > l) \leq K_2 \mu(U)l^{-1}
	\end{align}
	for all $ l>0 $. Then 
	\begin{align*}
	\Biggl(\int_{U_\delta} |u|^{\alpha_0} \Biggr)^\frac{1}{\alpha_0} d \mu \leq C_\textup{BG} \mu(U)^\frac{1}{\alpha_0}
	\end{align*}
	where $ C_\textup{BG} $ depends only on $ K_1,\;K_2, \; \delta, \; \kappa $ and a lower bound on $ \alpha_0 $.
\end{lem}

For $ \delta \in (0,1), \; \tau >0, s \in \mathbb{R}  $ and $ r > 0 $ we denote
\begin{align}\label{def:paraball2}
    Q_- &= (s-(3+\delta)\tau r^2/4, s-(3-\delta)\tau r^2/4)\times B_W(x_0, \delta r),\\ \label{def:paraballminus} 
    Q_-' &= (s-\tau r^2, s-(3-\delta)\tau r^2/4)\times  B_W(x_0, \delta r), \\
    Q_+ &= (s-(1+\delta)\tau r^2/4, s)\times B_W(x_0, \delta r).\label{def:paraballplus}
    \end{align}

    We note that these are subsets of $Q$. Moreover we have $Q_- \subset Q_-'$. When we take suitable $\tilde{\delta} \in (1/2, 1)$, we have $Q_+ \subset Q_{\tilde{\delta}}$. (Recall that the definition of $Q_\delta$ $(\ref{def:paraball})$.)
    However, we can't compare $Q_-$ with $Q_{\tilde{\delta}}$ in general. To avoid this, we need to take another time-space ball $\tilde{Q}$ and consider $\tilde{Q}_{\tilde{\delta}}$. We will explain this in more detail later. (See the proof of parabolic Harnack inequality of Proposition \ref{prop:PHI}.)
\begin{lem}\label{lem:infL}
    Suppose that Assumptions \ref{asm:volRegiso} and \ref{asm:Dir} hold. Let $B = B_W(x_0, r)$, $p$, $\nu$ be as in Lemma \ref{lem:nu-1spaceTime}. Fix $ \tau > 0, \; \delta \in [1/2, 1) $ and $ \alpha_0 \in (0,\nu) $. Let $ u $ be any positive caloric function on $ Q = (s - \tau r^2, s) \times B_W(x,r) $.  Then
    \begin{align}
    \norm{u}_{\alpha_0,Q_-'} \leq C_3 \inf_{Q_+}u
    \end{align}
    where the constant $ C_3 $ depends on $d,\; \lambda,\; \Lambda, $  $ \tau, \; \nu, \; \alpha_0, \; \delta ,C_\textup{V},\;C_\textup{I},\; C_\textup{P}$.
\end{lem}
\begin{proof}
    Set $\mu = dt\otimes dx$  and take $ k = k(u,\kappa)  $ corresponding in $ \kappa = 1/2 $ in Lemma \ref{lem:CalLog}. Set $ v = e^ku $ and
	\begin{align*}
	&U = (s-\tau r^2, s-\tau r^2/2) \times B_W(x_0, r), \\ &U_\sigma = (s-\tau r^2, s-(3-\sigma)\tau r^2/4)\times B_W(x_0, \sigma r).
	\end{align*}
	By Lemma \ref{lem:supercalLL} and Assumption \ref{asm:volRegiso} it follows that
	\begin{align*}
	\norm{u}_{\alpha_0, U_{\sigma'}} \leq 
	\Biggl(C_3 \tau (1+\tau^{-1})^\frac{\nu}{\nu-1}\biggl( \frac{1}{(\sigma-\sigma')^2}  \biggr)^\frac{\nu}{\nu-1}  \Biggr)^{(1+\nu)(1/\alpha-1/\alpha_0) } \norm{u}_{\alpha,U_\sigma}
	\end{align*}
	for all $ 1/2 \leq \sigma' < \sigma \leq 1 $ and all $ \alpha \in (0, \alpha_0\nu^{-1}) $, in particular notice that $ \alpha_0\nu^{-1} > \alpha_0/2 $ and $ \alpha_0/2 < \nu/2 < 1 $ since $ \nu \in (1,2) $. 
    By Lemma \ref{lem:CalLog} we have that
	\begin{align*}
	\mu (\{(t,x) \in U \mid \log u > l \}) 
    \leq C_4' \mu(U)  \tau^{-1}\bigl(1 \vee \tau^2 \bigr) l^{-1},
	\end{align*}
    where $C_4'$ is a constant depending only on $C_4$ and $C_\textup{V}$. 
	Lemma \ref{lem:Bombieri-Giusti} is thus applicable and we obtain
    \begin{align*}
        \Biggl(\int_{Q_-'} (e^k u)^{\alpha_0} d\mu \Biggr)^\frac{1}{\alpha_0} \leq C_\textup{BG}\mu(U)^\frac{1}{\alpha_0} \leq C_\textup{BG}\mu(Q_-')^\frac{1}{\alpha_0},
    \end{align*}
    where $ C_\textup{BG} $ depends on $\lambda$, $\Lambda$, $\tau$,  $\nu$,  $\alpha_0$,  $d$,  $C_\textup{V}$, $C_\textup{I}$, $C_\textup{P}$. Dividing both side $ \text{by } \mu(Q_-')^\frac{1}{\alpha_0}$, we obtain
	\begin{align}\label{ine:BGresult1}
	\norm{e^ku}_{\alpha_0,Q_-'} \leq C_\textup{BG}.
	\end{align}

	On the other hand, we can now fix
	\begin{align*}
	V = (s - \tau r^2/2, s) \times B_W(x_0, r), \quad V_\sigma (s - (1+\sigma)\tau r^2/2, s) \times B_W(x_0, \sigma r)
	\end{align*}
	and apply Lemma \ref{lem:supercalInv} to $ w = e^{-k}u^{-1} $, this produces
	\begin{align*}
	\sup_{Q_{\sigma'}} w \leq C_1' \tau\biggl(\frac{1+\tau^{-1}}{(\sigma-\sigma')^2} \biggr)^\frac{\nu}{\nu-1}\norm{w}_{\alpha,Q_{\sigma}}^\alpha
	\end{align*}
	for all $ \alpha > 0 $ and $ 1/2 \leq \sigma< \sigma \leq 1 $. Since by Lemma \ref{lem:CalLog} we have
	\begin{align*}
	\mu \bigl(\{(t,x) \in V \mid \log w > l \} \bigr)
    \leq C_4'' \mu(V)  \tau^{-1}\bigl(1 \vee \tau^2 \bigr) l^{-1},
	\end{align*}
    where $C_4''$ is a constant depending only on $C_4$ and $C_\textup{V}$.
	Then Lemma \ref{lem:Bombieri-Giusti} for $ \alpha_0 = \infty $ is applicable and yields
	\begin{align}\label{ine:BGresult2}
	\sup_{Q_+} e^{-k}u^{-1} \leq C_\textup{BG} 
	\end{align}
	for some $ C_\textup{BG} $ which we can assume to be the same as before taking maximum of two. Multiplying $(\ref{ine:BGresult1})$ and $(\ref{ine:BGresult2})$ together, we have
    \begin{align*}
        \norm{e^ku}_{\alpha_0,Q_-'}\sup_{Q_+} e^{-k}u^{-1} \leq C_\textup{BG}^2.
    \end{align*} 
    Since $\sup_{Q_+} e^{-k}u^{-1} = \frac{1}{\inf_{Q_+}e^k u }$ and $k$ is the constant, we obtain the desired result.
\end{proof}

Now we can prove the parabolic Harnack inequality. 
 
\begin{prop}[Parabolic Harnack inequality]\label{prop:PHI}
	Suppose that Assumptions \ref{asm:volRegiso} and \ref{asm:Dir} hold. Let $B = B_W(x_0, r)$, $p$, $\nu$ be as in Lemma \ref{lem:nu-1spaceTime}. Fix $ \tau > 0 $ and $ \delta \in [1/2, 1) $. 
	Let $ u $ be  any positive caloric function in $ Q = (s - \tau r^2, s) \times B_W(x_0,r) $. Then we have
	\begin{align}
	\sup_{Q_-}u \leq C_\textup{H}\inf_{Q_+}u,
	\end{align}
	where the constant $ C_\textup{H} $ depends on $\lambda, \; \Lambda, \;  \tau, \; \nu, \; C_\textup{V}, \;C_\textup{I}, \;C_\textup{P}$, and $ \delta $.
\end{prop}

\begin{proof}
    We first note that $u$ is caloric in all open subsets of $Q$.
    Since $u$ is both subcaloric and supercaloric function, we will apply Corollary \ref{cor:subcalLL} and Lemma \ref{lem:infL} with suitable sets.
    Set $\tilde{s} = s-\frac{3-\delta}{4}\tau r^2$, $\tilde{\tau} = \frac{1+\delta}{4}\tau$, $\tilde{Q} = \tilde{Q}(\tilde{\tau}, x_0, \tilde{s}, r) = (\tilde{s}-\tilde{\tau}r^2, \tilde{s})\times B_W(x_0, r)$. 
    Observe that $\tilde{Q} = (s-\tau r^2, s-(3-\delta)\tau r^2/4)\times B_W(x_0, r) $. Since $\delta \geq 1/2$, we have $\frac{2\delta}{1+\delta} \geq \delta$ and $\tilde{Q}_{\frac{2\delta}{1+\delta}} \supset Q_{-}$. 
    Then applying Corollary \ref{cor:subcalLL} with $\tilde{Q}$, $\sigma=1$, $\sigma'=\frac{2\delta}{1+\delta}$, and $\alpha = 1$, we have
    \begin{align*}
        \sup_{Q_-}u \leq \sup_{\tilde{Q}_{\frac{2\delta}{1+\delta}}} u \leq \tilde{c} \norm{u}_{1, \tilde{Q}}, 
    \end{align*}
    where $\tilde{c}$ is a constant depending only on $d$, $\nu$, $C_\textup{V}$, $\tau$, and $\delta$.
    On the other hand, set $\hat{\tau} = \delta^2 \tau$, $\hat{r} = \frac{1}{\delta} r$, and $\hat{Q} = \hat{Q}(\hat{\tau}, x_0, s, \hat{r}) = (s-\hat{\tau}\hat{r}^2, s)\times B_W(x_0, \hat{r})$. A simple computation shows that $\tilde{Q} = \hat{Q}_-'  $ and $\hat{Q}_+ = (s-(1+\delta)\tau r^2/4, s) \times B_W(x_0, r) \supset Q_+$. Applying Lemma \ref{lem:infL} with $\hat{Q}$ and $\alpha_0 = 1$, we have 
    \begin{align*}
        \norm{u}_{1, \tilde{Q}} = \norm{u}_{1,\hat{Q}_-'} \leq \hat{c} \inf_{\hat{Q}_+} u \leq \hat{c}\inf_{Q_+}u,
    \end{align*} 
    where $\hat{c}$ is a constant depending only on $d,\; \lambda,\; \Lambda, $  $ \tau, \; \nu, \; \delta,\;C_\textup{V},\;C_\textup{I},\; C_\textup{P}$.
    Combining these inequalities, we obtain the desired result.
\end{proof}

Thanks to the parabolic Harnack inequality, the following H\"{o}lder continuity holds.

\begin{prop}\label{prop:hol}
    Suppose that Assumptions \ref{asm:volRegiso} and \ref{asm:Dir} hold. Let $R$, $B = B_W(x_0, r)$ be as in Lemma \ref{lem:nu-1spaceTime}. 
	Let $ \sqrt{t} \geq r $. Define $ t_0 := t + 1 $ and $ r_0 := \sqrt{t_0} $. If $ u $ is a positive caloric function on $ (0,t_0) \times B_W(x_0, r_0) $, then for all $ x, y \in B_W(x_0, r) $ we have
	\begin{align}
	u(t,x) - u(t,y) \leq c \biggl(\frac{r}{\sqrt{t}} \biggr)^\alpha \sup_{[3t_0/4, t_0]\times B_W(x_0, \sqrt{t_0}/2) }u,
	\end{align}
	where $ \alpha, \; c $ are constants which depend only on $ C_\textup{H} $.
\end{prop}

\begin{proof}
	Set $ r_k := 2^{-k}r_0 $ and let 
	\begin{align*}
	Q_k &:= (t_0 - r_k^2, t_0) \times B_W(x_0,r_k) , \\
	Q_k^- &:= (t_0 - 7/8r_k^2, t_0-5/8 r_k^2) \times B_W(x_0,\frac{1}{2}r_k) \\
	Q_k^+ &:= (t_0-1/4r_k^2, t_0)\times B_W(x_0, \frac{1}{2}r_k).
	\end{align*}
	Notice that $ Q_{k+1} \subset Q_k $ and $ Q_{k+1} = Q_k^+ $. We set
	\begin{align*}
	v_k = \frac{u - \inf_{Q_k}u}{\sup_{Q_k}u-\inf_{Q_k}u}.
	\end{align*}
	Then $ v_k $ is caloric on $ Q_k $, in particular $ 0 \leq v_k \leq 1 $ and 
	\begin{align*}
	\text{osc}(v_k,Q_k) := \sup_{Q_k} v_k - \inf_{Q_k} v_k = 1.
	\end{align*}
	Replacing $ v_k $ by $ 1 - v_k $ if we need, we can assume $ \sup_{Q_k^-} v_k \geq 1/2 $.  Now, for all $ k $ such that $ r_k \geq R^\frac{\theta}{d} $ we can apply the Parabolic Harnack inequality (Proposition \ref{prop:PHI}) and get
	\begin{align*}
	\frac{1}{2} \leq \sup_{Q_k^-} v_k \leq C_\textup{H} \inf_{Q_k^+} v_k.
	\end{align*}
	Since by construction $ Q_k^+ = Q_{k+1} $, we deduce that
	\begin{align*}
	\text{osc}(u,Q_{k+1}) &= \frac{\sup_{Q_{k+1}}u - \inf_{Q_{k+1}}u }{\text{osc}(u,Q_k)}\text{osc}(u,Q_k) \\
	&= \frac{\sup_{Q_{k+1}}u - \inf_{Q_{k+1}}(\inf_{Q_k}u + u - \inf_{Q_k}u) }{\text{osc}(u,Q_k)}\text{osc}(u,Q_k) \\
	&= \biggl( \frac{\sup_{Q_{k+1}}u - \inf_{Q_k}u  }{\text{osc}(u,Q_k)}- \inf _{Q_{k+1}}v_k \biggr)\text{osc}(u,Q_k) \\
	&\leq (1 - \inf_{Q_k^+}v_k)\text{osc}(u,Q_k) \\
	&\leq (1 - \delta)\text{osc}(u,Q_k).
	\end{align*}
	Here $ \delta^{-1} := 2C_\textup{H} $. We can now iterate the inequality up to $ k_0 $ such that $ r_{k_0} \geq r > r_{k_0+1} $ and obtain
	\begin{align*}
	\text{osc}(u,Q_{k_0}) \leq (1-\delta)^{k_0-1} \text{osc}(u,Q_1).
	\end{align*}
	Finally using the fact that $ B_W(x_0, r) \subset B_W(x_0, r_{k_0}) $, $ t \in (t_0 - r_{k_0}^2,t_0) $ and $ -k_0 \leq \log_2 (r/r_0) = \log_2 (r/\sqrt{t_0})$, for $x,y \in B_W(x_0, r)$ and $\sqrt{t} \geq r$ we have 
    \begin{align*}
        u(t,x) - u(t,y) 
        &\leq \text{osc}(u, Q_{k_0})
        \leq (1-\delta)^{k_0-1}\text{osc}(u, Q_1) \\
        &\leq (1-\delta)^{-1}\biggl(\frac{1}{1-\delta} \biggr)^{-k_0}\sup_{Q_1}u \\
        &\leq \ (1-\delta)^{-1}\biggl(\frac{1}{1-\delta} \biggr)^{\log_2 \frac{r}{\sqrt{t_0}}}\sup_{Q_1}u \\
        &= (1-\delta)^{-1}\biggl(2^{\log_2 \frac{r}{\sqrt{t_0}}} \biggr)^\frac{1}{1-\delta}\sup_{Q_1}u \\
        &= (1-\delta)^{-1} \biggl(\frac{t}{\sqrt{t_0}}\biggr)^\frac{1}{1-\delta}\sup_{Q_1}u
    \end{align*}
    which is the desired result since $Q_1 = [3t_0/4, t_0]\times B_W(x_0, \sqrt{t_0}/2) $.
\end{proof}

\begin{lem}\label{lem:stdHol}
	Suppose that Assumptions \ref{asm:volRegiso}--\ref{asm:regdensity} hold. Let $B = B_W(x_0, r)$ be as in Lemma \ref{lem:nu-1spaceTime} and $\sqrt{t}\geq r$. Then we have for almost all $x_0 \in W$
	\begin{align}
		\sup_{x,y \in B_W(x_0,r)}|p_t(0,x)-p_t(0,y)|\leq c\biggl(\frac{r}{\sqrt{t}}\biggr)^{\Holderind} t^{-\frac{d}{2}},
	\end{align}
	where ${\Holderind},c$ are positive constants depending only on $C_\textup{H}$.
\end{lem}
\begin{proof}
        Since $ p_t(0,\cdot) $ is caloric, from Proposition \ref{prop:hol} we have that 
        \begin{align*}
        p_t(0,x) - p_t(0,y) \leq c \biggl(\frac{r}{\sqrt{t}} \biggr)^\alpha \sup _{(s,z) \in [3t_0/4, t_0]\times B_W(x_0, \sqrt{t_0}/2)} p_s(0,z),
        \end{align*}
        for $x,y \in B_W(x_0, r)$, where $ t_0 = t+1 $. Thus we have just to bound the right-hand side of the above inequality. The parabolic Harnack inequality of Proposition \ref{prop:PHI} shows that
        \begin{align*}
        &\sup_{(s,z) \in [3t_0/4, t_0]\times B_W(x_0, \sqrt{t_0}/2 )}p_s(0,z)\\
        &\leq C_\textup{H} \inf_{(s,z) \in [3/2t_0, 7/4t_0]\times B_W(x_0, \sqrt{t_0}/2)} p_s(0,z) \\
        &\leq C_\textup{H} \inf_{z \in B_W(x_0, \sqrt{t_0}/2)} p_{\bar{t}}(0,z) \\
        &\leq C_\textup{H} |B_W(x_0, \sqrt{t}/2)|^{-1}\int_{B_W(x_0, \sqrt{t}/2)} p_{\bar{t}}(0,z)dz,
        \end{align*} 
        for $ \bar{t} \in [3/2t_0, 7/4t_0] $. Clearly $ \int_{B_W(x_0, \sqrt{t}/2)} p_{\bar{t}}(0,u)du \leq 1 $. 
        Hence we get the desired result.
\end{proof}

Under Assumption \ref{asm:holedist}, we can prove the following lemma. Recall that $g_W(x)$ is the first closest element of $x$ in $\overline{W}$.   

\begin{lem}\label{lem:asymdensity}
	Suppose that  Assumptions \ref{asm:volRegiso}--\ref{asm:holedist} hold. Let $ I \subset (0,\infty) $ be a compact interval. Then, for all $ R>0 $,
	\begin{align}\label{eqn:densityAsmp2}
	\lim_{r_0 \to 0}\limsup_{\epsilon \to 0}\sup_{\substack{x,y \in B_\textup{Euc}(0,R) \\ |x-y|<r_0 }}\sup_{t \in I}\epsilon^{-d}|p_{t/\epsilon^2}(0,g_W(x/\epsilon)) -p_{t/\epsilon^2}(0,g_W(y/\epsilon)) | = 0.
	\end{align}
\end{lem}
\begin{proof}
    Throughout the proof, we denote $g_W(x) = g(x)$ for simplicity.
	Let $z \in \overline{B_\textup{Euc}(0,R)}$.
	Set
    \begin{align*}
        \psi_\epsilon(r_0, R) = C_W\Biggl(\frac{2r_0}{\epsilon} + C_\text{hole}\biggl(\frac{2r_0+R}{\epsilon}\biggr)^\holeind\Biggr)\vee \biggl(\frac{2r_0}{\epsilon} + C_\text{hole}\biggl(\frac{2r_0+R}{\epsilon}\biggr)^\holeind\biggr)^\distind . 
    \end{align*} 
    Because of the definition of $g$ and Assumption \ref{asm:holedist}, we see that $x \in B_\textup{Euc}(z,2r_0)$ implies $g(x/\epsilon) \in B_W(g(z/\epsilon),\, \psi_\epsilon(r_0, R))$.
	Then the combination of the above and Lemma \ref{lem:stdHol} tells us that for all $\epsilon > 0$ 
     we have  
	\begin{align}
		&\sup_{x,y\in B_\textup{Euc}(z, 2r_0)}\epsilon^{-d}|p_{t/\epsilon^2}(0,g(x/\epsilon)) - p_{t/\epsilon^2}(0,g(y/\epsilon))| \nonumber 
        \leq c\biggl(\frac{\psi_\epsilon(r_0, R)}{\sqrt{t/\epsilon^2}}\biggr)^\alpha t^{-\frac{d}{2}} \nonumber \\
		&\leq c\Biggl(\frac{[2r_0 + \epsilon^{1-\holeind}C_\text{hole}(2r_0+R)]\vee \epsilon^{1-\distind} (2r_0+R)^\distind}{\sqrt{t}}\Biggr)^\Holderind t^{-\frac{d}{2}} .
	\end{align}
	Since $\Upsilon < 1$, we see that there exist $\hat{\epsilon} > 0$ and $c'>0$ such that for all $\epsilon \in (0, \hat{\epsilon})$ we have 
	\begin{align}\label{ine:pointsclHol}
		\sup_{x,y\in B_\textup{Euc}(z,2r_0)}\epsilon^{-d}|p_{t/\epsilon^2}(0,g(x/\epsilon)) - p_{t/\epsilon^2}(0,g(y/\epsilon))| \leq c'\biggl(\frac{r_0}{\sqrt{t}}\biggr)^\Holderind t^{-\frac{d}{2}}.
	\end{align}
    Note that we can choose $\hat{\epsilon}$ independently of $z$.
	Now let $t_1 = \inf_I t$. Fix $\delta >0$ and take $r_0>0$ small so that 
	\begin{align*}
		r_0 \leq\Biggl(\frac{t_1^\frac{d}{2}\delta}{c'}\Biggr)^\frac{1}{\Holderind}\hspace{-5pt}\cdot\sqrt{t_1} \wedge \biggl( \frac{\sqrt{t_1}}{2} \wedge 1 \biggr).
	\end{align*}
	Then the inequality  ($\ref{ine:pointsclHol}$) implies that
	\begin{align}
		\sup_{x,y\in B_\textup{Euc}(z,2r_0)}\epsilon^{-d}|p_{t/\epsilon^2}(0,g(x/\epsilon)) - p_{t/\epsilon^2}(0,g(y/\epsilon))|
        \leq \delta. \label{ine:pointwiseHol}
	\end{align}
	By the compactness of the set $\overline{B_\textup{Euc}(0,R)}$, we can take a finite cover $\{B_\textup{Euc}(z_i,\,r_0/2)\}_i$ of it. Moreover, we can assume that $\{z_i\}_i\subset B_\textup{Euc}(0,R)$. 
    Then the inequality $(\ref{ine:pointwiseHol})$ implies that for $\epsilon \in (0, \hat{\epsilon})$ we have
	\begin{align}\label{ine:unifHol}
		\sup_{i}\sup_{\substack{x, y\in B_\textup{Euc}(0,R) \\ x, y\in B_\textup{Euc}(z_i, 2r_0)}}\sup_{t\in I}\epsilon^{-d}|p_{t/\epsilon^2}(0,g(x/\epsilon)) - p_{t/\epsilon^2}(0,g(y/\epsilon))|  \leq \delta.
	\end{align}
	Now we take $x, y\in B_\textup{Euc}(0,R)$ so that $|x-y|<r_0$. Then there exists index $i$ such that $y \in B_\textup{Euc}(z_i,r_0/2) $. 
	Since $|x-y|<r_0$, we see that $x,y \in B_\textup{Euc}(z_i, 2r_0)$. Hence by $(\ref{ine:unifHol})$, we obtain
    \begin{align*}
		\sup_{t\in I} \epsilon^{-\frac{d}{2}}|p_{t/\epsilon^2}(0,g(x/\epsilon)) - p_{t/\epsilon^2}(0,g(y/\epsilon)) | \leq \delta.
	\end{align*}
	This implies that 
	\begin{align*}
		\limsup_{\epsilon \to 0}\sup_{\substack{x,y \in B_\textup{Euc}(0,R) \\ |x-y|<r_0 }}\sup_{t\in I} \epsilon^{-\frac{d}{2}}|p_{t/\epsilon^2}(0,g(x/\epsilon)) - p_{t/\epsilon^2}(0,g(y/\epsilon)) | \leq \delta.
	\end{align*}
	Since taking the limit $r_0 \to 0$ implies $\delta \to 0$, we have
	\begin{align*}
		\lim_{R_0\to 0} \limsup_{\epsilon \to 0}\sup_{\substack{x,y \in B_\textup{Euc}(0,R) \\ |x-y|<r_0 }}\sup_{t\in I} \epsilon^{-\frac{d}{2}}|p_{t/\epsilon^2}(0,g(x/\epsilon)) - p_{t/\epsilon^2}(0,g(y/\epsilon)) | = 0,
	\end{align*}
	which is the desired result. 
\end{proof}

\section{Local central limit theorem and the proof of the main theorem}
	In this section, we show the local CLT in a deterministic setting and apply it to typical realizations of the percolation cluster $W'(\omega)$ to prove Theorem \ref{thm:localCLT}. 
\begin{thm}\label{thm:deterministicLocalCLT}
	Suppose that Assumptions \ref{asm:volRegiso}--\ref{asm:invariance} hold.
	Fix a compact interval $ I \subset (0,\infty) $ and $ R > 0 $. Then we have
	\begin{align*} 
	\lim_{\epsilon \to 0} \sup_{x \in B_\textup{Euc}(0,R)}\sup_{t \in I}|\epsilon^{-d}p_{t/\epsilon^2}(0,g_W(x/\epsilon)) - k_t^{\Sigma}(x) | = 0.
	\end{align*}
\end{thm}

\begin{proof}
	Throughout the proof, we denote $g_W(x) = g(x)$ for simplicity.
	For $ x \in B_\textup{Euc}(0,R) $ and $ r > 0 $ set 
	\[J(t,\epsilon) = \frac{1}{\epsilon^d} \int_{B_\textup{Euc}(x,r)}p_{t/\epsilon^2} (0, g(y/\epsilon))  dy - \int _{B_\textup{Euc}(x,r)} k_t^\Sigma(y) dy .\]
	
	We divide $J$ into 3 parts; 
	\begin{align*}
	J(t,\epsilon) &= \frac{1}{\epsilon^d} \int_{B_\textup{Euc}(x,r)}\bigl(p_{t/\epsilon^2} (0,g(y/\epsilon)) -p_{t/\epsilon^2}(0,g(x/\epsilon)) \bigr)  dy \\
	&+ \frac{1}{\epsilon^d} \int_{B_\textup{Euc}(x,r)}\bigl(p_{t/\epsilon^2}(0,g(x/\epsilon)) - \epsilon^d k_t^\Sigma(x)\bigr)   dy \\ 
	&+ \int_{B_\textup{Euc}(x,r)}\bigl(k_t^\Sigma(x) - k_t^\Sigma(y)\bigr)dy \\
	&=: J_1(t,\epsilon) + J_2(t,\epsilon) + J_3(t,\epsilon) .
	\end{align*}
	We estimate $ J_1,J_3 $ and $ J $.
	
	\textit{Estimation of $J_3$.} By the continuity of $ k_t^\Sigma $, for each $\delta>0$, we can choose $ r \in (0,1) $ small enough such that 
	\begin{align}\label{ine:InProoflocalCLT1}
	\sup _{\substack{x,y \in B_\textup{Euc}(0,R+1) \\ |x-y|\leq r}} \sup _{t \in I} |k_t^\Sigma(y) - k_t^\Sigma(x) | \leq \delta,
	\end{align} 	
	which implies that $ \sup_{t\in I}|J_3(t,\epsilon)| \leq \delta |B_\textup{Euc}(x,r)| $.
	
	$  $
	
	\textit{Estimation of $J_1$.} By Lemma \ref{lem:asymdensity}, we can show that, for each $\delta>0$, we can find $ \bar{\epsilon} >0 $ and $r \in (0,1)$ such that for all $ \epsilon < \bar{\epsilon} $ 
	\begin{align}\label{ine:InProoflocalCLT2}
	\sup _{\substack{x,y \in B_\textup{Euc}(0,R+1) \\ |x-y|\leq r} } \sup _{t \in I} |p_{t/\epsilon^2} (0,g(y/\epsilon)) -p_{t/\epsilon^2}(0,g(x/\epsilon))| \leq \delta.
	\end{align}
	It follows that $ \sup_{t\in I}|J_1(t,\epsilon)| \leq \delta |B_\textup{Euc}(x,r)| $.
	
	$  $
	
	\textit{Estimation of $J$.}  By Assumption \ref{asm:invariance}, for each $\delta>0$, there is $\epsilon >0$ such that we have $ \sup_{t\in I}|J(t,\epsilon)| \leq \delta |B_\textup{Euc}(x,r)| $. 
	
	Now we estimate $ |J_2| $ uniformly in $ t $. Noting that $ |J_2(t,\epsilon)| \leq |J_1(t,\epsilon)| + |J_3(t,\epsilon)| +|J(t,\epsilon)| \leq 3\delta |B_\textup{Euc}(x,r)| $ and 
	\begin{align*}
	\sup _{t \in I}|J_2(t,\epsilon)| &= \sup_{t \in I} \biggl |\frac{1}{\epsilon^d}\int_{B_\textup{Euc}(x,r)}\bigl(p_{t/\epsilon^2}(0,g(x/\epsilon)) - \epsilon^d k_t^\Sigma(x)\bigr)  dy \biggr| \\
	&=  \sup_{t \in I}\biggr |\epsilon^{-d}p_{t/\epsilon^2}(0,g(x/\epsilon)) - k_t^\Sigma(x) \biggr|\cdot |B_\textup{Euc}(x,r)|, 
	\end{align*} 
	We get
	\begin{align*}
	\sup_{t \in I}\biggr |\epsilon^{-d}p_{t/\epsilon^2}(0,g(x/\epsilon)) - k_t^\Sigma(x) \biggr| 
	\leq  3\delta.
	\end{align*}
	
	This implies that
	\begin{align}\label{ine:point}
		\lim_{\epsilon \to 0}\sup_{t \in I}\biggr |\epsilon^{-d}p_{t/\epsilon^2}(0,g(x/\epsilon)) - k_t^\Sigma(x) \biggr| = 0.
	\end{align}
		
		Consider now $ R>0, \; \delta > 0 $ and let $ r \in (0,1) $ be chosen as above. Since $ B_\textup{Euc}(0,R) $ is compact, there exists a finite covering $ \{B_\textup{Euc}(z,r)\}_{z \in \mathcal{X}} $ of $ B_\textup{Euc}(0,R) $ with $ \mathcal{X} \subset B_\textup{Euc}(0,R) $. Since $ \mathcal{X} $ is finite, using $(\ref{ine:point})$ there exists $ \bar{\epsilon} > 0$ such that for all $ \epsilon \leq \bar{\epsilon} $ 
	\begin{align}\label{ine:InProoflocalCLT3}
	\sup_{z\in \mathcal{X}}\sup_{t \in I}|\epsilon^{-d}p_{t/\epsilon^2}(0,g(z/\epsilon))- k_t^\Sigma(z) | \leq \delta. 
	\end{align} 
	Next, for a general $ x \in B_\textup{Euc}(0,R) $, there exists some $ z \in \mathcal{X} $ such that $ x \in B_\textup{Euc}(z,r) $. Thus we can write 
	\begin{align*}
	\sup _{t \in I}|\epsilon^{-d} p_{t/\epsilon^2}(0,g(x/\epsilon)) - k_t^\Sigma(x) | &\leq \sup_{t \in I} |\epsilon^{-d} p_{t/\epsilon^2}(0,g(x/\epsilon)) - \epsilon^{-d} p_{t/\epsilon^2}(0,g(z/\epsilon)) | \\
	&+ \sup_{t \in I} |\epsilon^{-d} p_{t/\epsilon^2}(0,g(z/\epsilon)) - k_t^\Sigma(z) | \\
	&+ \sup_{t \in I}|a_\Lambda^{-1}k_t^\Sigma(z) - k_t^\Sigma(x)|.
	\end{align*} 
	Since $ x,z \in B_\textup{Euc}(0,R) $ and $ |x-z| \leq r $,  it follows from $(\ref{ine:InProoflocalCLT2})$ that the first term is bounded by $ \delta $. The second term is also bounded uniformly by $ \delta $ by means of $(\ref{ine:InProoflocalCLT3})$. Inequality $(\ref{ine:InProoflocalCLT1})$ implies that the last term is bounded by $ \delta $. Hence we obtain the desired result. 
\end{proof}

\noindent \textit{Proof of Theorem \ref{thm:localCLT}:} We apply Theorem \ref{thm:deterministicLocalCLT}. Then we only have to deduce Assumptions \ref{asm:volRegiso}--\ref{asm:invariance} from Assumptions \ref{asm:erg}--\ref{asm:volIso} for each $\HP$-almost all $\omega$.
Assumption \ref{asm:volRegiso} is directly deduced from Assumption \ref{asm:volIso}. Assumptions \ref{asm:Dir}--\ref{asm:regdensity} are also directly deduced from \ref{asm:ellipse}--\ref{asm:density}. 
Assumption \ref{asm:holedist} is a consequence of Assumption \ref{asm:holedistOfCluster}.
Assumption \ref{asm:invariance} is deduced from Theorem \ref{thm:qip}.

$ $\\
\textbf{Acknowledgements}
We would like to thank Professor Takashi Kumagai for suggesting the problem. We also would like to thank Professor Hideki Tanemura for fruitful discussions. We thank the referee for his numerous comments.
$ $

$ $


\end{document}